\documentclass[11pt, reqno]{amsart}
\usepackage{hyperref}
\hypersetup{
	colorlinks=true,
	linkcolor=blue,
	filecolor=blue,
	urlcolor=red,
	citecolor=green,}

\usepackage{subfigure}
\usepackage[graphicx]{realboxes}
\usepackage{tikz}
\usetikzlibrary {patterns}

\newtheorem{thm}{Theorem}[section]
\newtheorem{thm*}{Theorem A}
\newtheorem*{theorem*}{Theorem}
\newtheorem{cor}[thm]{Corollary}
\newtheorem{lem}[thm]{Lemma}

\newtheorem{rem}{Remark}

\theoremstyle{definition}
\newtheorem{defn}[thm]{Definition}
\theoremstyle{remark}

\numberwithin{equation}{section}

\usepackage{bm}
\newcommand{\thmref}[1]{Theorem~\ref{#1}}

\newcommand{\lemref}[1]{Lemma~\ref{#1}}

\newcommand{\corref}[1]{Corollory~\ref{#1}}

\newcommand{\R}{\mathbb{R}}

\newcommand{\Lie}{\mathcal{L}}

\newcommand{\ric}{{\rm Ric}}

\newcommand{\di}{{\rm div}}

\newcommand{\la}{\bm{\langle}}
\newcommand{\ra}{\bm{\rangle}}

\newcommand{\vol}{{\rm{vol}}}

\begin{document}
	
\title[Gradient Estimate for $\Delta_pv+bv^q+cv^r =0$]{Local and Global Log-Gradient estimates of solutions to $\Delta_pv+bv^q+cv^r =0$ on manifolds and applications}

%\title[Gradient Estimate for $\Delta_pv+bv^q+cv^r =0$]{Local and Global Log-Gradient estimates and Liouville Properties of solutions to $\Delta_pv+bv^q+cv^r =0$ on Riemannian manifolds}
\author{Jie He}
\address{School of Mathematics and Physics, Beijing University of Chemical Technology,  Chaoyang District, Beijing 100029, China}
	%    Current address
	%\curraddr{Department of Mathematics and Statistics, Case Western Reserve University, Cleveland, Ohio 43403}
\email{hejie@amss.ac.cn}

\author{Yuanqing Ma*}
\thanks{*Corresponding author}
\address{Institute of Geometry and Physics, University of Science and Technology of China, No. 96 	Jinzhai Road, Hefei, Anhui Province, 230026, China. }
\email{mayuanqinghaha@ustc.edu.cn }

\author{Youde Wang}
\address{1. School of Mathematics and Information Sciences, Guangzhou University; 2. Hua Loo-Keng Key Laboratory
		of Mathematics, Institute of Mathematics, Academy of Mathematics and Systems Science, Chinese Academy
		of Sciences, Beijing 100190, China; 3. School of Mathematical Sciences, University of Chinese Academy of Sciences,
		Beijing 100049, China.}
\email{wyd@math.ac.cn}
	
\begin{abstract}
In this paper, we employ the Nash-Moser iteration technique to study local and global properties of positive solutions to the equation $$\Delta_pv+bv^q+cv^r =0$$ on complete Riemannian manifolds with Ricci curvature bounded from below, where $b, c\in\mathbb R$, $p>1$, and $q\leq r$ are some real constants. Assuming certain conditions on $b,\, c,\, p,\, q$ and $r$, we derive succinct Cheng-Yau type gradient estimates for positive solutions, which is of sharp form. These gradient estimates allow us to obtain some Liouville-type theorems and Harnack inequalities. Our Liouville-type results are novel even in Euclidean spaces. Based on the local gradient estimates and a trick of Sung and Wang(\cite{MR3275651}), we also obtain the global gradient estimates for such solutions. As applications we show the uniqueness of positive solutions to some generalized Allen-Cahn equation and Fisher-KPP equation.
\end{abstract}
	%\thanks{Key words: gradient estimate; Nash-Moser iteration; Liouville type theorem}
	
\maketitle
\leftline{\quad Key words: Gradient estimate; Nash-Moser iteration; Liouville type theorem}
\leftline{\quad MSC 2020: 58J05; 35B45; 35J92}
	
\section{Introduction}
One trend in Riemannian geometry since the 1950's has been the study of how curvature affects global properties of partial differential equations and global quantities like the eigenvalues of the Laplacian. It is well-known that gradient estimate is a fundamental and powerful technique in the analysis of partial differential equations on Riemannian manifolds.
	
One of most classical results about gradient estimates can be traced back to Cheng-Yau's gradient estimate for positive harmonic functions (see \cite{MR1333601, MR431040}). Let $(M,g)$ be an $n$-dimensional complete non-compact Riemannian manifold with $\mathrm{Ric}_g\geq-(n-1)\kappa g$,  Cheng and Yau deduced that for any positive harmonic function $v$ in a geodesic ball $B_R(o)\subset M$, there holds
	\begin{align}\label{equa0}
		\sup_{B_{R/2}(o)}\frac{|\nabla v|}{v}\leq C_n\frac{1+\sqrt\kappa R}{R}.
	\end{align}
	An important feature of Cheng-Yau's estimate is that the right-hand side of \eqref{equa0} depends only on $n,\kappa$ and $R$, it does not depend on the injective radius or other global properties. Cheng-Yau's estimate turned out to be very useful. Harnack inequality follows immediately from Cheng-Yau's estimate. Liouville theorem for global positive harmonic functions on non-compact manifolds with non-negative Ricci curvature is also a direct consequence of  \eqref{equa0}.
	Cheng-Yau's approach can also be used to derive the estimates of the spectrum of manifolds  and investigate the geometry of manifolds(see \cite{MR1333601,MR2962229}). %\cite{MR2518892, MR2880214, MR3866881,MR3275651}.
	
After Cheng-Yau's work, gradient estimates for harmonic functions are generalized to more general space(\cite{MR3268873, MR2981845}) and gradient estimates for many other equations defined on Riemannian manifolds are established (see for example \cite{MR2880214,han2023gradient, hewangwei2024, MR0834612, Wang, MR4559367}). This topic has attracted the attention of many mathematicians.
	
In this article, we are concerned with the equation
\begin{equation}\label{equa1}
\begin{cases}
\Delta_pv+bv^q+cv^r  =0, &\quad \text{in}\, M;\\
v>0, &\quad \text{in}\, M;\\
b, c\in\mathbb R, ~p>1, ~q\leq r , &
\end{cases}
\end{equation}
where $(M,g)$ is a complete Riemannian manifold with Ricci curvature bounded from below. Here we require $q\leq r$ to eliminate possible duplicates since the two terms $bv^q$ and $cv^r$ are relatively symmetric. Equation \eqref{equa1} arises from many classical equations
and there are many questions related to equation (\ref{equa1}). For instances, in the case $b=0$ or $c=0$,  Eq \eqref{equa1} reduces to Lane-Emden equation; in the case $b=1, c=-1$, $p=2$, $q=1$, and $r=3$, Eq \eqref{equa1} is just Allen-Cahn equation; and in the case $b=1, c=-1$, $p=2, q=1, r=2$, Eq \eqref{equa1} is static Fisher-KPP equation, which has been largely studied in the last century. For more details we refer to \cite{MR4240763,fisher1937wave}.
	
The primary objectives of the present paper are twofold: Firstly, we want to establish exact Cheng-Yau type gradient estimates for positive solutions to equation \eqref{equa1} on a Riemannian manifold with Ricci curvature bounded below. Secondly, as application we would like to show some global properties of solutions to \eqref{equa1}, for instance, Liouville-type theorems. To achieve these goals, we employ the Nash-Moser iteration method and develop adaptive pointwise estimates for the linearized operator of the $p$-Laplacian operator.
	
When $b=1, c=0$, the equation (\ref{equa1}) reduces to the classical Lane-Emden-Fowler equation
	\begin{align}\label{equa:1.2}
		\begin{cases}
			\Delta_pv + v^q=0 & \text{in}\quad \R^n;\\
			v>0\quad & \text{in}\quad \R^n,
		\end{cases}
	\end{align}
where $q>0$ \text{and}  $1<p<n$. The equation (\ref{equa:1.2}) has already been the subject of countless publications (see \cite{MR1004713, MR1134481, MR982351, MR1121147, MR615628, MR829846, MR2350853, MR1946918, MR2522424}).

One of the questions solved (see \cite{MR1946918}) is that the Liouville property holds (i.e. equation (\ref{equa:1.2}) admits no positive solution) if and only if
	$$
	q\in \left(0,\,\, p^*-1 \right),
	$$
where $p^*=np/(n-p)$ is the Sobolev duality of $p$. In particular, in \cite{Li-Zhang} Y.Li and L.Zhang have ever established the following elegant Liouville theorem:

\begin{theorem*}(Li and Zhang \cite{Li-Zhang}) Let $g$ satisfy that $g$ is locally bounded and $s^{-\frac{n+2}{n-2}}g(s)$ is non-increasing. Assume that $u$ is a positive classical solution of equation $\Delta u + g(u)=0$ in $\mathbb{R}^n$ with $n \geq 3$, then either for some $b > 0$,
$$bu(x)= \left(\frac{\mu}{1 + \mu^2|x - \bar{x}|^2}\right)^{\frac{n-2}{2}}$$
or $u \equiv a$ for some $a > 0$ such that $g(a)=0$.
\end{theorem*}

It is worth to point out that the above theorem (due to Li-Zhang) readily implies that any positive $C^2$ solution to the Allen-Cahn equation on an Euclidean space $\mathbb{R}^n$ (with $n \geq 4$) or the static Fisher-KPP equation on $\mathbb{R}^n$ (with $n\geq 6$) must be constant $1$.

Now, let's recall some relevant work with the above equation \eqref{equa1} on a complete Riemannian manifold. In the sequel, we always assume that  $(M,g)$ is an n-dimension complete manifold with Ricci lower bound $\mathrm{Ric}_g\geq -(n-1)\kappa$.

In 2011, Wang and Zhang \cite{MR2880214} proved any positive $p$-harmonic function $v$ on  $(M,g)$ with Ricci curvature $\mathrm{Ric}_g\geq-(n-1)\kappa g$ satisfies
	\begin{align*}
		\sup_{B_{R/2}(o)}\frac{|\nabla v|}{v}\leq C_{n,p}\frac{1+\sqrt\kappa R}{R}.
	\end{align*}
Wang-Zhang's results generalized Cheng-Yau's result (\cite{MR431040}) of $p=2$ to any $p>1$ and improved Kotschwar-Ni's results \cite{MR2518892} by weakening sectional curvature condition to the Ricci curvature condition.
Later, Sung and Wang \cite{MR3275651} in 2014 studied the sharp gradient estimates for the eigenfunctions for $p$-Laplace operator. As a corollary, they proved that the optimal constant $C_{n,p}$ in the above estimate is $(n-1)/(p-1)$ if $v$ is a global solution.
	
Here we recall also some previous work on gradient estimates of positive solutions to Lane-Emden-Fowler equation on $(M,g)$, i.e.
\begin{align}\label{equa:1.3}
\begin{cases}
\Delta_pv + bv^q=0, \quad & \text{in}\,  M;\\
v>0,\quad & \text{in}\, M.
\end{cases}
\end{align}
In the case $p=2$ and $a>0$, Peng-Wang-Wei in 2020 derived some gradient estimates for the solutions to the equation. In particular, the following new estimate was obtained
\begin{align*}
\frac{\left\vert \nabla u \right\vert^{2}}{u^{2}}+au^{q-1}
\leq& \frac{2n}{2-n\max\{0,q-1\}}\left(\frac{C_{1}^{2}(n-1)(1+\sqrt{\kappa}R)+C_{2}}{R^{2}}\right.\\
&\left.+ 2\kappa+\frac{2nC_{1}^{2}}{(2+n\max\{0,q-1\})R^{2}}\right),
\end{align*}
if the Ricci curvature of domain manifold satisfies $\mathrm{Ric}_g\geq-(n-1)\kappa$ and $q<\frac{n+2}{n}$. Obviously, this is a stronger estimate than the logarithmic gradient estimate (also see \cite{HW, PWW1}).

Later, Wang-Wei \cite{MR4559367} derived Cheng-Yau type gradient estimates for positive solutions to \eqref{equa:1.3} under the assumption
	$$
	q\in \left(-\infty,\quad \frac{n+1}{n-1}+\frac{2}{\sqrt{n(n-1)}}\right).
	$$
For any $p>1$ and $a>0$, He-Wang-Wei \cite{hewangwei2024} proved that  the Cheng-Yau type gradient estimate holds for positive solutions of \eqref{equa:1.3} when
	$$
	q\in \left(-\infty,\quad \frac{n+3}{n-1}(p-1)\right).
	$$
We refer to \cite{MR3912761, MR3866881} when the $p$-Laplacian operator in the equation \eqref{equa:1.3} is replaced by weighted $p$-Laplacian Lichnerowicz operator.

In \cite{WZ1} Wang and A. Zhang considered also the gradient estimates on positive solutions to the following elliptic equation defined on a complete Riemannian manifold $(M,\,g)$:
$$\Delta v+v^r-v^s= 0,$$
where $r$ and $s$ are two real constants. The equation is just Eq. \eqref{equa1} with $p=2, b=1$ and $c=-1$. When $(M,\,g)$ satisfies $\mathrm{Ric}_g\geq -(n-1)\kappa$ (where $n\geq 2$ is the dimension of $M$ and $\kappa$ is a nonnegative constant), they employed the Nash-Moser iteration technique to derive a Cheng-Yau's type gradient estimate for positive solution to the above equation under some suitable geometric and analysis conditions.

%Moreover, it is shown that when the Ricci curvature of $M$ is nonnegative, this elliptic equation does not admit any positive solution except for $u\equiv 1$ if $r \leq s$ and $$1<r<\frac{n+3}{n-1}\quad\quad ~~\mbox{or}~~\quad 1<s<\frac{n+3}{n-1}.$$	

Inspired by the previous work, we will combine the point-wise estimate and the Nash-Moser iteration technique to establish an exact Cheng-Yau logarithmic gradient estimate for positive solutions to equation
	$$\Delta_pv+bv^q+cv^r=0$$
on a complete Riemannian manifold with Ricci curvature bounded from below. Our main theorem is stated as follows.
	
\begin{thm}\label{thm1}
Let $(M^n,g)$ be an $n$-dimensional complete Riemannian manifold with Ricci curvature satisfying $\ric_g\geq -(n-1)\kappa g$ for some constant $\kappa\geq 0$. Suppose that $p>1$ and $v$ is a $C^1$-positive solution of \eqref{equa1} on a geodesic ball $B_R(o)\subset M$. If $(b,c,q, r)\in W_1\cup W_2 \cup W_3$, where $W_1$, $W_2$ and $W_3$   are defined respectively by
\begin{align}\label{cond1a}
W_1=&\left\{(b, c, q,r): b\left(\frac{n+1}{n-1}-\frac{q}{p-1}\right)\geq 0 \quad \text{and}\quad c\left(\frac{n+1}{n-1}-\frac{r}{p-1}\right)\geq 0\right\},
\\
\label{cond2a}
W_2=&\left\{(b, c, q,r): 	c\leq 0\quad\text{and}\quad \left|\frac{q}{p-1}-\frac{n+1}{n-1}\right|< \frac{2}{n-1}\right\},
\\
\label{cond3a}
	W_3=&\left\{(b, c, q,r): b\geq 0\quad \text{and}\quad \left|\frac{r}{p-1}-\frac{n+1}{n-1}\right|< \frac{2}{n-1}\right\},
\end{align}
then there holds
\begin{align}\label{equa:1.8}
\sup_{B_{R/2}(o)} \frac{|\nabla v|}{v}\leq C(n,p,q,r)\frac{1+\sqrt{\kappa}R}{R},
\end{align}
where the constant $C(n,p,q,r)$ depends only on $n$, $p$, $q$ and $r$ and $C(n,p,q,r)$ corresponding to different $W_i$ ($i=1,2,3$) may be different from each other.
\end{thm}
	
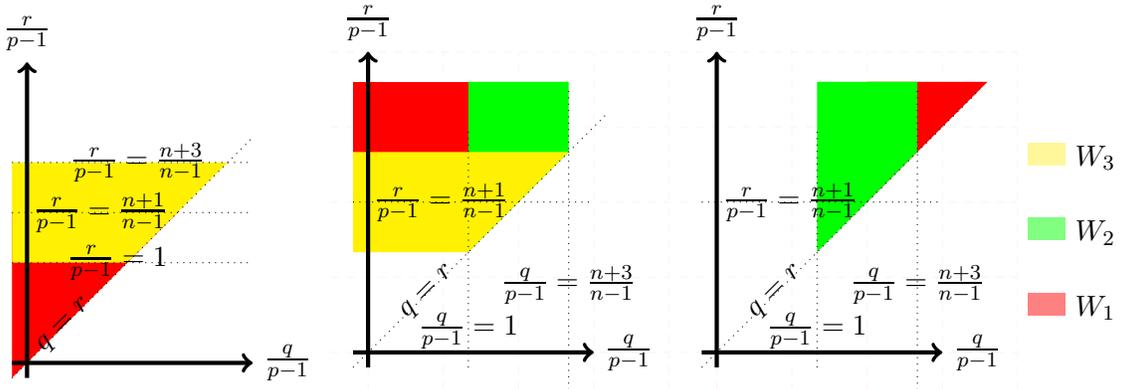
\begin{figure}[h]
		\begin{tikzpicture}
		\pgfsetfillopacity{0.5}
		%\draw[help lines, color=red!5, dashed] (-0.5,-0.5) grid (3.2,3.2);
		\path[fill=red](-0.2, -0.2)--(2,2)--(-0.2, 2);
		%\path[fill=green](4/3,3.6)--(4/3,4/3)--(8/3,8/3)--(8/3, 3.6);
		\path[fill=yellow](4/3, 4/3)--(8/3,8/3)--(-0.2,8/3)--(-0.2, 4/3);
		%\fill[pattern=north west lines](3,3) circle (1);
		%\fill[gray](3,3) circle (1);
		%\path[pattern=north west lines](2, 2)--(7, 2)--(7, 5/2)--(5 ,5/2)--(5,5)--(2, 5);
		%\draw[pattern=north west lines] (1, 1) rectangle (6 ,5);
		%\draw[pattern=north west lines] (1, 1) rectangle (7 ,6/5);
		%\draw[pattern=fivepointed stars] (0.5,0.5) rectangle (3,1);
		\draw[dotted] (-0.2, -0.2) -- (3, 3);
		%\draw[dotted] (-0.2, 8/3 )--(2, 8/3) node[anchor=east]{$\frac{r}{p-1}= \frac{n+3}{n-1}$}-- (4, 8/3 );
		\draw[dotted] (-0.2, 2 )--(2, 2) node[anchor=east]{$\frac{r}{p-1}= \frac{n+1}{n-1}$}-- (3, 2 );
		%\draw[dotted] (-0.5, 4/3 )--(2, 4/3) node[anchor=east]{$\frac{r}{p-1}= 1 $}-- (4, 4/3 )  -- (5, 4/3);
		\draw[dotted] (-0.2, 8/3 )--(2.5, 8/3) node[anchor=east]{$\frac{r}{p-1}= \frac{n+3}{n-1}$}-- (3, 8/3 );
		\draw[dotted] (-0.2, 4/3 )--(2, 4/3) node[anchor=east]{$\frac{r}{p-1}= 1$}-- (3, 4/3);
		%\draw[dotted] (-0.5,2) --(1,2)  node[anchor=north]{$\frac{r}{p-1}= 1$} -- (4,2)
		%\path[fill=yellow](-0.5, 1.8)--(3.5,1.8)--(3.5,-0.5)--(-0.5, -0.5);
		\draw (0.5,0.5) node[] {\rotatebox{45}{$q=r $}};
		\draw[->,ultra thick] (-0.2,0)--(3,0) node[right]{$\frac{q}{p-1}$};
		\draw[->,ultra thick] (0,-0.2)--(0,4) node[above]{$\frac{r}{p-1}$};
	\end{tikzpicture}			
	\begin{tikzpicture}
		\pgfsetfillopacity{0.5}
		\draw[help lines, color=red!5, dashed] (-0.5,-0.5) grid (4,4);
		\path[fill=red](-0.2, 2)--(2,2)--(2,3.6)--(-0.2, 3.6);
		\path[fill=green](4/3,3.6)--(4/3,4/3)--(8/3,8/3)--(8/3, 3.6);
		\path[fill=yellow](4/3, 4/3)--(8/3,8/3)--(-0.2,8/3)--(-0.2, 4/3);
		%\fill[pattern=north west lines](3,3) circle (1);
		%\fill[gray](3,3) circle (1);
		\draw[dotted] (-0.2, -0.2) -- (3.2,3.2);
		\draw[dotted] (4/3,-0.2) --(4/3, 2/3)  node[anchor=north]{$\frac{q}{p-1}=1 $} -- (4/3,8/3) -- (4/3,3);
		%\draw[dotted] (-0.2, 8/3 )--(2, 8/3) node[anchor=east]{$\frac{r}{p-1}= \frac{n+3}{n-1}$}-- (4, 8/3 );
		\draw[dotted] (-0.2, 2 )--(2, 2) node[anchor=east]{$\frac{r}{p-1}= \frac{n+1}{n-1}$}-- (3, 2 );
		\draw[dotted] (8/3,-0.5) --(8/3, 1.3)  node[anchor=north]{$\frac{q}{p-1}= \frac{n+3}{n-1}$} -- (8/3,3.5);
		\draw (0.8,0.8) node[] {\rotatebox{45}{$q=r $}};
		\draw[->,ultra thick] (-0.2,0)--(3,0) node[right]{$\frac{q}{p-1}$};
		\draw[->,ultra thick] (0,-0.2)--(0,4) node[above]{$\frac{r}{p-1}$};
	\end{tikzpicture}		
	\begin{tikzpicture}
		\pgfsetfillopacity{0.5}
		\draw[help lines, color=red!5, dashed] (-0.5,-0.5) grid (4,4);
		\path[fill=red](2, 2)--(3.6,3.6)--(2,3.6);
		\path[fill=green](4/3,4/3)--(8/3,8/3)--(8/3,3.6)--(4/3, 3.6);
		%\fill[pattern=north west lines](3,3) circle (1);
		%\fill[gray](3,3) circle (1);
		%\path[pattern=north west lines](2, 2)--(7, 2)--(7, 5/2)--(5 ,5/2)--(5,5)--(2, 5);
		%\draw[pattern=north west lines] (1, 1) rectangle (6 ,5);
		%\draw[pattern=north west lines] (1, 1) rectangle (7 ,6/5);
		%\draw[pattern=fivepointed stars] (0.5,0.5) rectangle (3,1);
		\draw[dotted] (-0.2, -0.2) -- (3.5,3.5);
		\draw[dotted] (4/3,-0.2) --(4/3, 2/3)  node[anchor=north]{$\frac{q}{p-1}=1 $} -- (4/3,8/3) -- (4/3,3);
		%\draw[dotted] (-0.2, 8/3 )--(2, 8/3) node[anchor=east]{$\frac{r}{p-1}= \frac{n+3}{n-1}$}-- (4, 8/3 );
		\draw[dotted] (-0.2, 2 )--(2, 2) node[anchor=east]{$\frac{r}{p-1}= \frac{n+1}{n-1}$}-- (3, 2 );
		\draw[dotted] (8/3,-0.5) --(8/3, 1.3)  node[anchor=north]{$\frac{q}{p-1}= \frac{n+3}{n-1}$} -- (8/3,3.5);
		%\draw[dotted] (2,-0.5) --(2, 2)  node[anchor=north]{$\frac{q}{p-1}= \frac{n+1}{n-1}$} -- (2,4);
		%\filldraw[black] (3,3) circle (1pt) node[anchor=south]{$(\frac{n+1}{n-1}, \frac{n+1}{n-1})$};
		%\draw[->] (3, 4.5) node[anchor=south]{$	1<\frac{r}{p-1}<\frac{n+3}{n-1},  1<\frac{q}{p-1}<\frac{n+2}{n},$}--(2.35, 3.2);
		%\path[fill=yellow](-0.5, 1.8)--(3.5,1.8)--(3.5,-0.5)--(-0.5, -0.5);
		\draw (0.8,0.8) node[] {\rotatebox{45}{$q=r $}};
		\draw[->,ultra thick] (-0.2,0)--(3,0) node[right]{$\frac{q}{p-1}$};
		\draw[->,ultra thick] (0,-0.2)--(0,4) node[above]{$\frac{r}{p-1}$};
	\end{tikzpicture}				
	\begin{tikzpicture}
		\path[fill=red!0](0, 1)--(0.5,1)--(0.5,1.1) --(0.5,1.2)--(0,1.2);		
		\path[fill=red!50](0, 2)--(0.5,2)--(0.5,2.1) node[anchor=west]{$W_1$}--(0.5,2.3)--(0,2.3);				
		\path[fill=green!50](0, 3)--(0.5,3)--(0.5,3.1) node[anchor=west]{$W_2$}--(0.5,3.3)--(0,3.3);
		\path[fill=yellow!50](0,4)--(0.5,4)--(0.5,4.1) node[anchor=west]{$W_3$}--(0.5,4.3)--(0,4.3);
	\end{tikzpicture}
	\caption{The region of $(\frac{q}{p-1}, \frac{r}{p-1})$ in $W_1, W_2, W_3$ with different signs of $b,c$.The first figure shows the case $b>0,c>0$; the second figure shows the case $b>0, c<0$; the third figure shows the case $b<0, c<0$.}
\end{figure}
	
\begin{rem} In the case $p=2$, $b=1$ and $c=-1$, Wang and Zhang employed in \cite{WZ1} the same method to approach the equation as in present paper. But, they did not consider the case $(q, r)\in W_1$, where $W_1$ is the same as in \thmref{thm1}, i.e.
\begin{align}
W_1=\left\{(q,r):\,\, \frac{n+1}{n-1}-q\geq 0 \quad \text{and} \quad \frac{n+1}{n-1}-r\leq 0\right\}.
\end{align}
\end{rem}

\begin{rem}
In the case $c=0$, \thmref{thm1} implies that  if
$$
b>0\quad\text{and}\quad q<\frac{n+3}{n-1}(p-1);
$$
or,
$$
 b<0 \quad\text{and}\quad q>p-1,
$$
then any positive solution of $\Delta_pv+bv^q=0$ satisfies the estimate \eqref{equa:1.8}, which covers the main theorem in \cite{hewangwei2024}.
\end{rem}

By the gradient estimate obtained above, we can get a Liouville-type results for such positive solutions on Riemannian manifolds with non-negative Ricci curvature.

\begin{thm}\label{thm1.4}
Let $(M^n,g)$ be an $n$-dimensional complete non-compact Riemannian manifold with non-negative Ricci curvature. Assume one of conditions \eqref{cond1a}, \eqref{cond2a}, and \eqref{cond3a} is fulfilled. In addition, we assume $q\neq r$ and $bv^q+cv^r\not\equiv0$ to avoid some duplicates. Then we have
\begin{itemize}
	\item If $bc>0$, then \eqref{equa1} admits no positive solutions on $M$;
	\item If $b>0$ and $c<0$, then the \eqref{equa1} has the only trivial solution $v=(-\frac{c}{b})^{\frac{1}{q-r}}$;
	\item If $b=0$ and $c\neq0$; or, $b\neq 0$ and $c=0$, then \eqref{equa1} admits no positive solutions on $M$.
\end{itemize}
\end{thm}
	
By \thmref{thm1}, it is not hard to get the Harnack inequality which is a natural corollary of the gradient estimate.
	
\begin{thm}\label{thm1.5}
Let $(M^n,g)$ be a complete non-compact Riemannian manifold with Ricci curvature $\ric_g\geq-(n-1)\kappa g$, where $\kappa$ is a non-negative constant. Suppose
$p>1$ and $v\in C^1(B_R(o))$ is a positive solution to equation \eqref{equa1}, defined on a geodesic ball $B_R(o)\subset M$, with constants $b,c,q,r$ satisfying \eqref{cond1a} or \eqref{cond2a} or \eqref{cond3a}. Then, for any $x, y\in B_{R/2}(o)$ there holds
		$$
		v(x)/v(y)\leq  e^{C(n,p,q,r)(1+\sqrt{\kappa}R)}.
		$$
Moreover, if $v\in C^1(M)$ is an entire positive solution to equation \eqref{equa1} with constants $b,c,q,r$ satisfying \eqref{cond1a} or \eqref{cond2a} or \eqref{cond3a} in $M$, then, for any $x,\, y\in M$ there holds
		$$
		v(x)/v(y)\leq  e^{C(n,p,q,r)\sqrt{\kappa}d(x,y)},
		$$
where $d(x,y)$ is the geodesic distance between $x$ and $y$.
\end{thm}

If $v$ is a global positive solution of \eqref{equa1}, by the local gradient estimate, we can bound the gradient globally. Moreover, employing a trick due to Sung and Wang in \cite{MR3275651}, we can give an explicit expression of global universal bound for gradient of positive solution to \eqref{equa1}.

\begin{thm}\label{t10}
Let $(M^n,g)$ be a complete non-compact Riemannian manifold with $\mathrm{Ric}_g\geq-(n-1)\kappa g$ where $\kappa\geq 0$, and $v$ be a global positive solution of \eqref{equa1} in $(M, g)$. The regions $W_i$ ($i=1, 2, 3$) are the same as that in \thmref{thm1}.
\begin{enumerate}
\item If $(b, c, q, r)\in W_1$, then we have
\begin{align*}
\frac{|\nabla v|}{v}(x)\leq \frac{(n-1)\sqrt{\kappa}}{p-1}, \quad\forall x\in M.
\end{align*}
		
\item If $(b, c, q, r)\in W_2$, then we have
\begin{align*}
\frac{|\nabla v|}{v}(x)\leq \frac{2\sqrt{\kappa}}{(p-1)\sqrt{\left(\frac{q}{p-1}-1\right)\left(\frac{n+3}{n-1}-\frac{q}{p-1}\right)}}, \quad\forall x\in M.
\end{align*}
\item If $(b, c, q, r)\in W_3$, then we have
\begin{align*}
\frac{|\nabla v|}{v}(x)\leq \frac{2\sqrt{\kappa}}{(p-1)\sqrt{\left(\frac{r}{p-1}-1\right)\left(\frac{n+3}{n-1}-\frac{r}{p-1}\right)}}, \quad\forall x\in M.
\end{align*}

\end{enumerate}
\end{thm}

As direct consequences of the above theorem we have the following corollaries:
\begin{cor}\label{thm1.6}
Let $(M^n,g)$ be a complete non-compact Riemannian manifold with $\mathrm{Ric}_g\geq-(n-1)\kappa g$ where $\kappa\geq 0$, and $v$ be an entire positive solution to the generalized Allen-Cahn equation with $p>1$
	$$\Delta_pv + v - v^3=0.$$
We have
\begin{enumerate}
\item if $p$ satisfies
\begin{equation}\label{equa1.10}
\frac{2n}{n+1}\leq p \leq \frac{4n-2}{n+1},
\end{equation}
then there holds
\begin{align*}
\frac{|\nabla v|}{v}(x)\leq \frac{(n-1)\sqrt{\kappa}}{p-1}, \quad\forall x\in M.
\end{align*}
\item if $p$ satisfies
\begin{equation}\label{equa1.11}
\frac{4n-2}{n+1}< p <4,
\end{equation}
then there holds
\begin{align*}
\frac{|\nabla v|}{v}(x)\leq \frac{2\sqrt{\kappa}}{(p-1)\sqrt{\left(\frac{3}{p-1}-1\right)\left(\frac{n+3}{n-1}-\frac{3}{p-1}\right)}}, \quad\forall x\in M.
\end{align*}

\item If $p$ satisfies
\begin{equation}\label{equa1.12}
\frac{2n+2}{n+3}< p <\frac{2n}{n+1},
\end{equation}
then there holds
\begin{align*}
\frac{|\nabla v|}{v}(x)\leq \frac{2\sqrt{\kappa}}{(p-1)\sqrt{\left(\frac{1}{p-1}-1\right)\left(\frac{n+3}{n-1}-\frac{1}{p-1}\right)}}, \quad\forall x\in M.
\end{align*}
\end{enumerate}

In particular, if $\kappa=0$ and
$$
\frac{2n+2}{n+3}< p < 4,
$$
then $v\equiv 1$ is the unique positive solution.
\end{cor}

\begin{cor}\label{thm1.7}
Let $(M^n,g)$ be a complete non-compact Riemannian manifold with Ricci curvature $\mathrm{Ric}_g\geq-(n-1)\kappa g$ where $\kappa\geq 0$, and $v$ be an entire positive solution to the generalized Fisher-KPP equation with $p>1$
$$\Delta_pv + v - v^2=0.$$
\begin{enumerate}
\item If $p$ satisfies
\begin{equation}\label{equa1.13}
\frac{2n}{n+1}\leq p \leq \frac{3n-1}{n+1},
\end{equation}
then there holds
\begin{align*}
\frac{|\nabla v|}{v}(x)\leq \frac{(n-1)\sqrt{\kappa}}{p-1}, \quad\forall x\in M.
\end{align*}
\item If $p$ satisfies
\begin{equation}\label{equa1.14}
\frac{3n-1}{n+1}< p <3,
\end{equation}
then there holds
\begin{align*}
\frac{|\nabla v|}{v}(x)\leq \frac{2\sqrt{\kappa}}{(p-1)\sqrt{\left(\frac{2}{p-1}-1\right)\left(\frac{n+3}{n-1}-\frac{2}{p-1}\right)}}, \quad\forall x\in M.
\end{align*}
\item If $p$ satisfies
\begin{equation}\label{equa1.15}
	\frac{2n+2}{n+3}< p <\frac{2n}{n+1},
\end{equation}
then there holds
\begin{align*}
	\frac{|\nabla v|}{v}(x)\leq \frac{2\sqrt{\kappa}}{(p-1)\sqrt{\left(\frac{1}{p-1}-1\right)\left(\frac{n+3}{n-1}-\frac{1}{p-1}\right)}}, \quad\forall x\in M.
\end{align*}
\end{enumerate}

In particular, if $\kappa=0$ and
$$
\frac{2n+2}{n+3}< p < 3,
$$
then $v\equiv 1$ is the unique positive solution.
\end{cor}

The structure of the paper is organized as follows. In Section 2, we give a meticulous estimate of $\mathcal{L} \left(|\nabla \log v|^{2\alpha}\right)$ where $\mathcal{L}$ is the linearized operator of the $p$-Laplacian operator at $u= -(p-1)\ln v$ (see \eqref{linearization} for the explicit definition of the operator $\mathcal{L}$). We also recall the Saloff-Coste's Sobolev embedding theorem in section 2. In section 3, we establish a universal integral estimate on $|\nabla \log v|^{2\alpha}$ and then use delicately the Nash-Moser iteration to prove main results of this paper. The proof of Theorem \ref{t10}, Corollary \ref{thm1.6} and Corollary \ref{thm1.7} is provided in Section 4.

\section{preliminary}
	
In this paper, $(M,g)$ is an $n$-dimensional Riemannian manifold and $\nabla$ is the associated Levi-Civita connection. For any real function $\varphi\in C^1(M)$, we denote $\nabla \varphi\in \Gamma(T^*M)$ by $\nabla \varphi(X)=\nabla_X\varphi$. The volume form is denoted as $$\vol=\sqrt{\det(g_{ij})}d x_1\wedge\ldots\wedge dx_n$$ where $(x_1,\ldots, x_n)$ is a local coordinate. For simplicity, we omit the volume form  in  integrations.
	
For $p>1$, the $p$-Laplacian operator is defined by
	$$
	\Delta_pu:=\di\left(|\nabla u|^{p-2}\nabla u\right).
	$$
The solution of $p$-Laplacian equation $\Delta_pu=0$ is the critical point of the energy functional
	$$
	E(u)=\int_M|\nabla u|^p.
	$$
	
\begin{defn}\label{def1}
A function $v$ is said to be a positive (weak) solution of equation \eqref{equa1} on a region $\Omega$ if $v\in C^1(\Omega), v>0 $ and for all $\psi\in C_0^\infty(\Omega)$, we have
\begin{align*}
-\int_\Omega|\nabla v|^{p-2}\la\nabla v,\nabla\psi\ra+\int_\Omega (bv^q+cv^r)\psi =0.
\end{align*}
\end{defn}
	
It is worth mentioning that any solution $v$ of equation (\ref{equa1}) satisfies $v\in W^{2,2}_{loc}(\Omega)$ and $v\in C^{1,\alpha}(\Omega)$ for some $\alpha\in(0,1)$(for example, see \cite{MR0709038, MR0727034,MR0474389}). Moreover, $v$ is in fact smooth away from $\{\nabla v=0\}$.
	
Next, we recall Saloff-Coste's Sobolev inequalities (see \cite[Theorem 3.1]{saloff1992uniformly}) which shall play an important role in our proof of the main theorem.
	
\begin{lem}[\cite{saloff1992uniformly}]\label{salof}
Let $(M^n,g)$ be a complete manifold with $\mathrm{Ric}_g\geq-(n-1)\kappa$ where $\kappa$ is a non-negative constant. For $n>2$, there exists a positive constant $c_0$ depending only on $n$, such that for all $B\subset M$ of radius R and volume $V$ we have for $f\in C^{\infty}_0(B)$
		$$
		\|f\|_{L^{\frac{2n}{n-2}}(B)}^2\leq e^{c_0(1+\sqrt{\kappa}R)}V^{-\frac{2}{n}}R^2\left(\int_B|\nabla f|^2+R^{-2}f^2\right).
		$$
For $n=2$, the above inequality holds if we replace $n$ by some fixed $n'>2$.
\end{lem}
	
For any positive solution $v$ to the equation (\ref{equa1}), if we take a logarithmic transformation
$$v = e^{-\frac{u}{p-1}},$$
then, it is easy to see that $u$ satisfies the following equation
\begin{align}\label{equa2.1}
\Delta_pu-|\nabla u|^p-b_1e^{q_1u}-c_1e^{r_1u}=0,
\end{align}
where
\begin{align}\label{defofbc}
b_1 =b(p-1)^{p-1} ,\quad c_1=c(p-1)^{p-1},\quad q_1=1-\frac{q}{p-1}\quad\text{and} \quad r_1=1-\frac{r}{p-1}.
\end{align}
	
Denote $f=|\nabla u|^2$. It is clear that (\ref{equa2.1}) becomes
\begin{align}\label{equa2.3}
\Delta_pu-f^\frac{p}{2}-b_1e^{q_1u}-c_1e^{r_1u}=0.
\end{align}
Now we consider the linearization operator $\Lie$ of $p$-Laplacian:
\begin{align}\label{linearization}
\Lie(\psi)=\di\left(f^{p/2-1}A\left(\nabla \psi\right)\right),
\end{align}
where
\begin{align}\label{defofA}
A(\nabla\psi) = \nabla\psi+(p-2)f^{-1}\la\nabla \psi,\nabla u\ra\nabla u.
\end{align}

We also need to recall a useful expression of $\mathcal L(f^\alpha)$ for any $\alpha>0$ as follows:
\begin{lem}[\cite{hewangwei2024}]\label{lem2.3}
For any $\alpha >0$, the following identity holds point-wisely in $\{x : f(x) > 0\}$,
\begin{align}\label{bochner1}
\begin{split}
\mathcal{L} (f^{\alpha}) = &\alpha(\alpha+\frac{p}{2}-2)f^{\alpha+\frac{p}{2}-3}|\nabla f|^2 + 2\alpha f^{\alpha+\frac{p}{2}-2} \left(|\nabla\nabla u|^2 + \ric(\nabla u,\nabla u) \right)\\
&+\alpha(p-2)(\alpha-1)f^{\alpha+\frac{p}{2}-4}\langle\nabla f,\nabla u\rangle^2 + 2\alpha f^{\alpha-1}\langle\nabla\Delta_p u,\nabla u\rangle.
\end{split}
\end{align}
\end{lem}

\section{Proof of Main Theorem}
We divide the proof of \thmref{thm1} into three parts. In the first part, we derive a fundamental integral inequality on $f=|\nabla u|^2$, which will be used in the second and third parts. In the second part, we give an $L^{\beta}$-estimate of $f$ on a geodesic ball with radius $3R/4$, where $L^{\beta}$ norm of $f$ determines the initial state of the Nash-Moser iteration. Finally, we give a complete proof of our theorem by an intensive use of Nash-Moser iteration method.
	
\subsection{Integral inequality}\label{sect:3.1}\

In order to establish the main theorem, we need to show some lemmas on integral estimates of $f$. We first introduce the following point-wise estimate for $\Lie(f^\alpha)$.

\begin{lem}\label{lem3.1}
Let $(M^n,g)$ be an $n$-dimensional complete Riemannian manifold with Ricci curvature $\ric_g\geq -(n-1)\kappa g$ where $\kappa$ is a non-negative constant. Assume that $v$ is a positive solution to the equation (\ref{equa1}), $u = -(p-1)\log v$ and $f=|\nabla u|^2$. If one of conditions \eqref{cond1a}, \eqref{cond2a}, and \eqref{cond3a} is satisfied, then there exists constant $\alpha_0=\alpha_0(n,p,q,r)>0$ such that the following inequality
\begin{align*}
\begin{split}
\mathcal{L}(f^{\alpha_0 })\geq& 2\alpha_0 \beta_{n,p,q,r} f^{\alpha_0 +\frac{p}{2}} -2\alpha_0 (n-1)\kappa f^{\alpha_0 +\frac{p}{2}-1}- \alpha_0  a_1|\nabla f|f^{\alpha_0 +\frac{p-3}{2}}
\end{split}
\end{align*}
holds point-wisely in $\{x: f(x)>0\}$, where $a_1 = \left|p -\frac{2(p-1)}{n-1}\right|$. Here $\alpha_0$ corresponding to different $W_i$ ($i=1, 2, 3$) maybe different from each other.
\end{lem}
	
\begin{proof}
Let $\{e_1,e_2,\ldots, e_n\}$ be  an orthonormal frame of $TM$ on a domain with $f\neq 0$ such that $e_1=\frac{\nabla u}{|\nabla u|}$. We have $u_1 = f^{1/2}$. Direct computation shows
\begin{align}\label{equ2.8}
u_{11} = \frac{1}{2}f^{-1/2}f_1 = \frac{1}{2}f^{-1}\la\nabla u,\nabla f\ra,
\end{align}
and
\begin{align}\label{equa2.9}
|\nabla f|^2/f=4\sum_{i=1}^n u_{1i}^2 \geq 4u^2_{11}.
\end{align}
It follows from (\ref{equa2.3}) and (\ref{equ2.8}) that
\begin{align}\label{equ2.10}
\begin{split}
				\langle\nabla\Delta_p u,\nabla u\rangle=
				&\ \langle\nabla u, \nabla(f^\frac{p}{2}+ b_1e^{q_1u}+c_1e^{r_1u}) \rangle\\
				=&\ pf^{\frac{p}{2} }u_{11}+ \left(b_1q_1e^{q_1u}+c_1r_1e^{r_1u}\right)f.
\end{split}
\end{align}
By Cauchy inequality, we arrive at
\begin{align}\label{equ2.11}
|\nabla\nabla u|^2 \geq& u_{11}^2 + \sum_{i=2}u_{ii}^2 \geq u_{11}^2 +\frac{1}{n-1}\left(\sum_{i=2}u_{ii}\right)^2.
\end{align}
		
From now on, we always assume that $\alpha\geq 3/2$, which can guarantee $\alpha+p/2-1>0$. Substituting (\ref{equ2.8}), (\ref{equa2.9}), (\ref{equ2.10}) and (\ref{equ2.11}) into (\ref{bochner1}) and collecting terms about $u_{11}^2$, we obtained
\begin{align}\label{equa3.5}
\begin{split}
\frac{f^{2-\alpha-\frac{p}{2}}}{2\alpha}\mathcal{L} (f^{\alpha}) \geq &\ (2\alpha-1)(p-1) u_{11}^2 + \ric(\nabla u,\nabla u)+\frac{1}{n-1}\left(\sum_{i=2}u_{ii}\right)^2\\
&+ pf u_{11}+\left(b_1q_1e^{q_1u}+c_1r_1e^{r_1u}\right)f^{2-\frac{p}{2}}.
\end{split}
\end{align}
If we express the $p$-Lapalcian in terms of $f$, we have
\begin{align}\label{equ:2.12}
\Delta_p u =&f^{\frac{p}{2}-1}\left((p-1)u_{11}+\sum_{i=2}^nu_{ii}\right).
\end{align}
Substituting (\ref{equ:2.12}) into equation (\ref{equa2.3}), we obtain
\begin{align}\label{eq:2.13}
(p-1)u_{11}+\sum_{i=2}^nu_{ii}^2 = f + (b_1e^{q_1u}+c_1e^{r_1u})f^{1-\frac{p}{2}}.
\end{align}
It follows from (\ref{eq:2.13}) that
\begin{align}\label{equ:2.13}
\begin{split}
\frac{1}{n-1}\left(\sum_{i=2}u_{ii}\right)^2=&\frac{1}{n-1}\left(f+(b_1e^{q_1u}+c_1e^{r_1u})f^{1-\frac{p}{2}} -(p-1)u_{11}\right)^2\\
=&\frac{1}{n-1}\Big(f^2+(b_1e^{q_1u}+c_1e^{r_1u})^2f^{2-p}+(p-1)^2u^2_{11} \\
&+2(b_1e^{q_1u}+c_1e^{r_1u})f^{2-\frac{p}{2}} -2(p-1)fu_{11} \\
&-2(p-1)(b_1e^{q_1u}+c_1e^{r_1u})f^{1-\frac{p}{2}}u_{11}\Big).
\end{split}
\end{align}
Hence, combining (\ref{equa3.5}) and (\ref{equ:2.13}) together leads to
\begin{align}\label{equa2.15}
\begin{split}
\frac{f^{2-\alpha-\frac{p}{2}}}{2\alpha}\mathcal{L} (f^{\alpha})\geq & (2\alpha-1)(p-1) u_{11}^2 + \frac{f^2}{n-1}+\ric(\nabla u,\nabla u)			
+\left(p -\frac{2(p-1)}{n-1}\right)f u_{11}\\
&+\frac{(b_1e^{q_1u}+c_1e^{r_1u})^2f^{2-p}}{n-1}-\frac{2(p-1)(b_1e^{q_1u}+c_1e^{r_1u}) }{n-1}f^{1-\frac{p}{2}} u_{11} \\
& +b_1\left(q_1+\frac{2}{n-1}\right)e^{q_1u}f^{2-\frac{p}{2}}+c_1\left(r_1+\frac{2}{n-1}\right)e^{r_1u}f^{2-\frac{p}{2}}.
\end{split}
\end{align}
Making use of the fact $a^2+2ab\geq -b^2$, we have
\begin{align}\label{equa2.16}
\begin{split}
& (2\alpha-1)(p-1)  u_{11}^2-\frac{2(p-1) (b_1e^{q_1u}+c_1e^{r_1u}) }{n-1}f^{1-\frac{p}{2}} u_{11}\\
\geq &- \frac{(p-1) \left(b_1e^{q_1u}+c_1e^{r_1u}\right)^2}{  (2\alpha-1)(n-1)}f^{2-p}.
\end{split}
\end{align}
Substituting (\ref{equa2.16}) into (\ref{equa2.15}) yields
\begin{align}\label{equa:3.11}
\begin{split}
\frac{f^{2-\alpha-\frac{p}{2}}}{2\alpha}\mathcal{L} (f^{\alpha}) \geq&\
\frac{f^2}{n-1} + B_{n,p,\alpha}\left(b_1e^{q_1u}+c_1e^{r_1u}\right)^2f^{2-p} + \ric(\nabla u,\nabla u)	-\frac{a_1}{2}f^{\frac{1}{2}} |\nabla f|\\
& + b_1\left(q_1+\frac{2}{n-1}\right)e^{q_1u}f^{2-\frac{p}{2}} + c_1\left(r_1+\frac{2}{n-1}\right)e^{r_1u}f^{2-\frac{p}{2}},
\end{split}
\end{align}
where we have used \eqref{equ2.8} to estimate $fu_{11}$ and the constants $a_1$ and $B_{n,p,\alpha}$ are defined respectively by
\begin{align*}
a_1=\left|p -\frac{2(p-1)}{n-1}\right|
\end{align*}
and
\begin{align}\label{equa:3.12}
B_{n,p,\alpha} = \frac{1}{n-1}-\frac{p-1 }{  (2\alpha-1)(n-1) }.
\end{align}
For any $n>1$ and $p>1$, we can always choose $\alpha$ large enough such that $B_{n,p,\alpha}$ is positive.
		
\textbf{Case 1.}
		$$
		 b\left(\frac{n+1}{n-1}-\frac{q}{p-1}\right)\geq 0 \quad \text{and}\quad c\left(\frac{n+1}{n-1}-\frac{r}{p-1}\right)\geq 0.
		$$
For this case, by the definition of $b_1$ and $c_1$ in \eqref{defofbc} we have
$$
b_1\left(q_1+\frac{2}{n-1}\right) \geq 0\quad\text{and}\quad c_1\left(r_1+\frac{2}{n-1}\right) \geq 0.
$$
Hence,
\begin{align*}
\frac{f^{2-\alpha-\frac{p}{2}}}{2\alpha}\mathcal{L} (f^{\alpha}) \geq&\ \frac{f^2}{n-1} + B_{n,p,\alpha}\left(b_1e^{q_1u}+c_1e^{r_1u}\right)^2f^{2-p}\\
&+ \ric(\nabla u,\nabla u)-\frac{a_1}{2}f^{\frac{1}{2}} |\nabla f|.
\end{align*}
From \eqref{equa:3.12}, we can see
		$$
		\lim_{\alpha\to\infty}B_{n,p,\alpha}=\frac{1}{n-1}>0.
		$$
So there exists a constant $\alpha_1=\alpha_1(n,p)\geq 3/2$ such that, for any  $\alpha\geq \alpha_1$, there hold $B_{n,p,\alpha}>0$ and
		\begin{align}\label{equa:3.13a}
			\frac{f^{2-\alpha-\frac{p}{2}}}{2\alpha}
			\mathcal{L} (f^{\alpha}) \geq&
			\frac{f^2}{n-1}
			+ \ric(\nabla u,\nabla u)
			-\frac{a_1}{2}f^{\frac{1}{2}} |\nabla f|.
		\end{align}

\textbf{Case 2.}
$$
	c\leq 0 \quad\text{and}\quad \left|\frac{q}{p-1}-\frac{n+1}{n-1}\right|< \frac{2}{n-1}
$$
In the present situation, we can conclude from (\ref{equa:3.11}) that
\begin{align*}
\frac{f^{2-\alpha-\frac{p}{2}}}{2\alpha}\mathcal{L} (f^{\alpha}) \geq&\
B_{n,p,\alpha}\left(b_1e^{q_1u}+c_1e^{r_1u}\right)^2f^{2-p} + \left(q_1+\frac{2}{n-1}\right)\left(b_1e^{q_1u}+c_1e^{r_1u}\right)f^{2-\frac{p}{2}}\\
& +\frac{f^2}{n-1} + \ric(\nabla u,\nabla u) -\frac{a_1}{2}f^{\frac{1}{2}} |\nabla f| +c_1\left(r_1-q_1\right)e^{r_1u}f^{2-\frac{p}{2}}.
\end{align*}
The conditions $q\leq r$ and $c\leq 0$ imply
\begin{align*}
c_1\left(r_1-q_1\right)e^{r_1u}f^{2-\frac{p}{2}}=c(p-1)^{p-1}\left(q-r\right)e^{r_1u}f^{2-\frac{p}{2}}\geq0.
\end{align*}
By the basic inequality $A^2+2AB\geq -B^2$, we can conclude
\begin{align*}
&B_{n,p, \alpha}\left(b_1e^{q_1u}+c_1e^{r_1u}\right)^2f^{2-p} + \left(q_1+\frac{2}{n-1}\right)\left(b_1e^{q_1u}+c_1e^{r_1u}\right)f^{2-\frac{p}{2}}\\
\geq &-\frac{1}{4B_{n,p, \alpha}}\left(q_1+\frac{2}{n-1}\right)^2f^2.
\end{align*}
It follows
\begin{align*}
\frac{f^{2-\alpha-\frac{p}{2}}}{2\alpha}\mathcal{L} (f^{\alpha})
\geq&\ \beta_{n,p,q, \alpha} f^2-(n-1)\kappa f-\frac{a_1}{2}f^{\frac{1}{2}} |\nabla f|,
\end{align*}
where
$$
\beta_{n,p,q,\alpha}:=\frac{1}{n-1}-\frac{1}{4B_{n,p, \alpha}}\left(q_1+\frac{2}{n-1}\right)^2 .
$$
By the conditions given for this case, we know that
$$
\lim_{\alpha\to\infty}\left\{\frac{1}{n-1}-\frac{1}{4B_{n,p,q,\alpha}}\left(q_1+\frac{2}{n-1}\right)^2\right\}
=\frac{1}{n-1}-\frac{(n-1)}{4}\left(\frac{q}{p-1}-\frac{n+1}{n-1}\right)^2>0.
$$
So, there exists a constant $\alpha_2>0$ such that, for any  $\alpha\geq \alpha_2$, there hold $\beta_{n,p,q, \alpha}>0$ and
\begin{align}\label{equa:3.14a}
	\frac{f^{2-\alpha-\frac{p}{2}}}{2\alpha}
	\mathcal{L} (f^{\alpha}) \geq&
	\beta_{n,p,q, \alpha}f^2
	+ \ric(\nabla u,\nabla u)
	-\frac{a_1}{2}f^{\frac{1}{2}} |\nabla f|.
\end{align}

\textbf{Case 3.}
$$
b\geq 0\quad \text{and}\quad \left|\frac{r}{p-1}-\frac{n+1}{n-1}\right|<\frac{2}{n-1}
$$
For this case, we can conclude from (\ref{equa:3.11}) that
\begin{align*}
	\frac{f^{2-\alpha-\frac{p}{2}}}{2\alpha}
	\mathcal{L} (f^{\alpha}) \geq&\
	B_{n,p, \alpha}\left(b_1e^{q_1u}+c_1e^{r_1u}\right)^2f^{2-p}
	+ \left(r_1+\frac{2}{n-1}\right)\left(b_1e^{q_1u}+c_1e^{r_1u}\right)f^{2-\frac{p}{2}}
	\\
	&
	+\frac{f^2}{n-1}
	+ \ric(\nabla u,\nabla u)				
	-\frac{a_1}{2}f^{\frac{1}{2}} |\nabla f|
	+b_1\left(q_1-r_1\right)e^{r_1u}f^{2-\frac{p}{2}}.
\end{align*}
The basic conditions $q\leq r$ and $b\geq 0$ imply
\begin{align*}
b_1\left(q_1-r_1\right)e^{r_1u}f^{2-\frac{p}{2}}=b(p-1)^{p-1}\left(r-q\right)e^{r_1u}f^{2-\frac{p}{2}}\geq0.
\end{align*}
By a similar procedure as in Case 2, we can conclude
\begin{align*}
	&B_{n,p, \alpha}\left(b_1e^{q_1u}+c_1e^{r_1u}\right)^2f^{2-p}
	+ \left(r_1+\frac{2}{n-1}\right)\left(b_1e^{q_1u}+c_1e^{r_1u}\right)f^{2-\frac{p}{2}}\\
	\geq &-\frac{1}{4B_{n,p, \alpha}}\left(r_1+\frac{2}{n-1}\right)^2f^2.
\end{align*}
Now we use the two conditions in this case to obtain
\begin{align*}
		\frac{f^{2-\alpha-\frac{p}{2}}}{2\alpha}
		\mathcal{L} (f^{\alpha})
		\geq&\
		\beta_{n,p, r,\alpha} f^2
		-(n-1)\kappa f
		-\frac{a_1}{2}f^{\frac{1}{2}} |\nabla f|,
\end{align*}
where
$$
\beta_{n,p, r,\alpha}:=\frac{1}{n-1}-\frac{1}{4B_{n,p, \alpha}}\left(r_1+\frac{2}{n-1}\right)^2.
$$
Similar to Case 2, there exists a large $\alpha_3>0$ such that, for any  $\alpha\geq \alpha_3$, there hold $\beta_{n,p, r,\alpha}>0$ and
\begin{align}\label{equa:3.15a}
	\frac{f^{2-\alpha-\frac{p}{2}}}{2\alpha}
	\mathcal{L} (f^{\alpha}) \geq&
	\beta_{n,p, r,\alpha}f^2
	+ \ric(\nabla u,\nabla u)
	-\frac{a_1}{2}f^{\frac{1}{2}} |\nabla f|.
\end{align}

From now on, we denote
		$$
		\alpha_0 =
		\begin{cases}
		 \alpha_1, \quad (b, c, q, r)\in W_1;\\
		 \alpha_2, \quad (b, c, q, r)\in W_2\setminus W_1;\\
		 \alpha_3, \quad (b, c, q, r)\in W_3\setminus (W_1\cup W_2),
		\end{cases}
		$$
and
		$$
		\beta_{n,p,q,r} =
		\begin{cases}
			\frac{1}{n-1}, &\quad (b, c, q, r)\in W_1;\\
			\beta_{n,p,q,\alpha_2}, &\quad (b, c, q, r)\in W_2\setminus W_1;\\
			\beta_{n,p,r,\alpha_3}, &\quad (b, c, q, r)\in W_3\setminus  W_1 .
		\end{cases}
		$$
We summarize the above four cases to obtain the following point-wise estimate
\begin{align}\label{equa:3.15}
\begin{split}
\mathcal{L} (f^{\alpha_0 })\geq& 2\alpha_0 \beta_{n,p,q,r} f^{\alpha_0 +\frac{p}{2} }-2\alpha_0 (n-1)\kappa  f^{\alpha_0 +\frac{p}{2}-1}
- \alpha_0  a_1   |\nabla f|f^{\alpha_0 +\frac{p-3}{2}}.
\end{split}
\end{align}
Thus, we complete the proof of this lemma.
\end{proof}
	
\begin{lem}\label{lem3.2}
Let $(M^n,g)$ be an $n$-dim complete Riemannian manifold with $\ric_g\geq -(n-1)\kappa g$ where $\kappa$ is a non-negative constant and $\Omega = B_R(o)\subset M$ be a geodesic ball. Assume that $v$ is a positive solution to the equation (\ref{equa1}), $u = -(p-1)\ln v$ and $f=|\nabla u|^2$. If one of conditions \eqref{cond1a}, \eqref{cond2a}, and \eqref{cond3a} is satisfied, then there exist constants $\beta_{n,p,q,r}$, $a_3,\, a_4$ and $a_5$, which depend only on $n$, $p,\,q$, and $r$, such that for any $t\geq t_0 $ where $t_0$ is defined in \eqref{equa2.8} there holds true
\begin{align*}
\begin{split}
& \beta_{n,p,q,r}\int_\Omega f^{\alpha_0+\frac{p}{2}+t}\eta^2 + \frac{a_3}{ t }e^{-t_0}V^{\frac{2}{n}}R^{-2}\left\|f^{\frac{\alpha_0+t-1}{2}+\frac{p}{4} }\eta\right\|_{L^{\frac{2n}{n-2}}}^2\\
\leq\  & a_5t_0^2R^{-2} \int_\Omega f^{\alpha_0+\frac{p}{2}+t-1}\eta^2+\frac{a_4}{t }\int_\Omega f^{\alpha_0+\frac{p}{2}+t-1}|\nabla\eta|^2,
\end{split}
\end{align*}
where $\eta\in C^{\infty}_0(\Omega,\R)$.
\end{lem}
	
\begin{proof}
By the regularity theory on elliptic equations, $u$ is smooth away from $\{f=0\}$. So the both sides of \eqref{equa:3.15} are in fact smooth. Let $\epsilon>0$ and $\psi = f_\epsilon^{\alpha}\eta^2 $, where $f_\epsilon = (f-\epsilon)^+$, $\eta\in C^{\infty}_0(B_R(o))$ is non-negative and $\alpha>1$ which will be determined later. Now, we multiply the both sides of \eqref{equa:3.15} by $\psi$, integrate then the obtained inequality over $\Omega$ and take a direct computation to get
\begin{align}\label{equa:3.17}
\begin{split}
				&-\int_\Omega \alpha_0 tf^{\alpha_0+\frac{p}{2}-2}f_{\epsilon}^{t-1}|\nabla f|^2\eta^2
				+
				t\alpha_0(p-2)f^{\alpha_0+\frac{p}{2}-3}f_{\epsilon}^{t-1}\la\nabla f,\nabla u\ra^2\eta^2
				\\
				&-\int_\Omega2\eta\alpha_0 f^{\alpha+\frac{p}{2} -2}f_{\epsilon}^t\la\nabla f,\nabla\eta\ra+2\alpha_0\eta(p-2)f^{\alpha+\frac{p}{2}-3}f_{\epsilon}^t\la\nabla f,\nabla u\ra\la\nabla u,  \nabla\eta\ra
				\\
				\geq &
				2\beta_{n,p,q,r}\alpha_0 \int_\Omega f^{\alpha_0+\frac{p}{2}}f_{\epsilon}^t\eta^2  -2(n-1)\alpha_0\kappa\int_\Omega f^{\alpha_0+\frac{p}{2}-1}f_{\epsilon}^t\eta^2
				- a_1\alpha_0\int_\Omega f^{\alpha_0+\frac{p-3}{2}}f_{\epsilon}^t|\nabla f|\eta^2.
\end{split}
\end{align}
It is easy to see that
		\begin{align}\label{2.24}
			f_{\epsilon}^{t-1}|\nabla f|^2 +(p-2)f_{\epsilon}^{t-1}f^{-1}\la\nabla f,\nabla u\ra^2\geq a_2 f_{\epsilon}^{t-1}|\nabla f|^2,
		\end{align}
where $a_2 = \min\{1, p-1\}$ and
\begin{align}\label{2.25}
f_{\epsilon}^t\la\nabla f,\nabla\eta\ra+ (p-2)f_{\epsilon}^tf^{-1}\la\nabla f,\nabla u\ra\la\nabla u,  \nabla\eta\ra\geq -(p+1)f_{\epsilon}^t |\nabla f||\nabla\eta|.
\end{align}
Substituting \eqref{2.24}) and \eqref{2.25} into \eqref{equa:3.17} and letting $\epsilon\to 0$, we obtain
		\begin{align}
			\label{2.26}
			\begin{split}
				&2\beta_{n,p,q,r}\int_\Omega f^{\alpha_0+\frac{p}{2}+t}\eta^2
				+
				\int_\Omega a_2  tf^{\alpha_0+\frac{p}{2}+t-3}|\nabla f|^2\eta^2\\
				\leq &~
				2(n-1) \kappa\int_\Omega f^{\alpha_0+\frac{p}{2}+t-1}\eta^2
				+ a_1 \int_\Omega f^{\alpha_0+\frac{p-3}{2}+t }|\nabla f|\eta^2
				\\
				&
				+2 (p+1)\int_\Omega  f^{\alpha_0+\frac{p}{2}+t-2}|\nabla f||\nabla\eta|\eta,
		\end{split}\end{align}
where we have divided the both sides of the inequality by $\alpha_0$.
		
Using Cauchy's inequality, we have
		\begin{align}\label{2.27}
			\begin{split}
				a_1 f^{\alpha_0+\frac{p-3}{2}+t }|\nabla f|\eta^2\leq &\
				\frac{a_2t}{4}    f^{\alpha_0+\frac{p}{2}+t-3}|\nabla f|^2\eta^2
				+\frac{ a_1^2}{a_2t} f^{\alpha_0+\frac{p}{2}+t}\eta^2,\\
			\end{split}
		\end{align}
and
		\begin{align}\label{2.27*}
			\begin{split}
				2(p+1)f^{\alpha_0+\frac{p}{2}+t-2}|\nabla f||\nabla\eta|\eta\leq &\
				\frac{a_2t}{4}    f^{\alpha_0+\frac{p}{2}+t-3}|\nabla f|^2\eta^2
				+\frac{4(p+1)^2 }{a_2t} f^{\alpha_0+\frac{p}{2}+t-1}|\nabla \eta|^2.
			\end{split}
		\end{align}
Now we choose $t$ large enough such that
		\begin{align}\label{2.28}
			\frac{a_1^2}{a_2t}\leq  \beta_{n,p,q,r}.
		\end{align}
It follows from (\ref{2.26}), (\ref{2.27}), \eqref{2.27*} and (\ref{2.28}) that
		\begin{align}
			\label{2.29}
			\begin{split}
				& \beta_{n,p,q,r}\int_\Omega f^{\alpha_0+\frac{p}{2}+t}\eta^2
				+
				\frac{a_2  t}{2}\int_\Omega f^{\alpha_0+\frac{p}{2}+t-3}|\nabla f|^2\eta^2 \\
				\leq  &\
				2(n-1) \kappa\int_\Omega f^{\alpha_0+\frac{p}{2}+t-1}\eta^2
				+\frac{4(p+1)^2 }{a_2t}  \int_\Omega f^{\alpha_0+\frac{p}{2}+t-1}|\nabla \eta|^2.
			\end{split}
		\end{align}

On the other hand, we have
\begin{align}\label{2.30}
\begin{split}
\left|\nabla \left(f^{\frac{\alpha_0+t-1}{2}+\frac{p}{4} }\eta \right)\right|^2\leq &\ 2\left|\nabla f^{\frac{\alpha_0+t-1}{2}+\frac{p}{4} }\right|^2\eta^2 +2f^{\alpha_0+t-1+\frac{p}{2}}|\nabla\eta|^2\\
=&\ \frac{(\alpha_0+t+\frac{p}{2}-1)^2}{2}f^{\alpha_0+t+\frac{p}{2}-3}|\nabla f |^2\eta^2 +2f^{\alpha_0+t-1+\frac{p}{2}}|\nabla\eta|^2  .
\end{split}
\end{align}
Now, substituting \eqref{2.30} into \eqref{2.29} leads to
\begin{align}\label{2.31}
\begin{split}
& \beta_{n,p,q,r} \int_\Omega f^{\alpha_0+\frac{p}{2}+t}\eta^2 + \frac{4a_2t}{(2\alpha_0+2t+p-2)^2}\int_\Omega \left|\nabla \left(f^{\frac{\alpha_0+t-1}{2}+\frac{p}{4} }\eta\right)\right|^2 \\
\leq  &\ 2(n-1)\kappa  \int_\Omega f^{\alpha_0+t+\frac{p}{2}-1}\eta^2 + \frac{4(p+1)^2 }{a_2t} \int_\Omega f^{\alpha_0+\frac{p}{2}+t-1}|\nabla\eta|^2\\
&\ +\frac{8a_2t}{(2\alpha_0+2t+p-2)^2}\int_\Omega f^{\alpha_0+t+\frac{p}{2}-1}|\nabla\eta|^2 .
\end{split}
\end{align}

From now on we use $a_1,\, a_2,\, a_3\, \cdots$ to denote constants depending only on $n,\, p,\, q$ and $r$. We choose $a_3$ and $a_4$ such that
\begin{align}\label{2.32}
\frac{a_3}{t}\leq  \frac{4a_2t}{(2\alpha_0+2t+p-2)^2}\quad\text{and}\quad\frac{8a_2t}{(2\alpha_0+2t+p-2)^2}+\frac{4(p+1)^2}{a_2t} \leq\frac{a_4}{t},
\end{align}
since $t$ satisfies (\ref{2.28}). From (\ref{2.31}) and (\ref{2.32}) we can obtain that
\begin{align}\label{2.33}
\begin{split}
& \beta_{n,p,q,r}\int_\Omega f^{\alpha_0+\frac{p}{2}+t}\eta^2 + \frac{a_3}{t}\int_\Omega   \left|\nabla \left(f^{\frac{\alpha_0+t-1}{2}
+\frac{p}{4} }\eta\right)\right|^2 \\
\leq  &\ 2(n-1)\kappa  \int_\Omega f^{\alpha_0+t+\frac{p}{2}-1}\eta^2 + \frac{a_4 }{t} \int_\Omega f^{\alpha_0+\frac{p}{2}+t-1}\left|\nabla\eta\right|^2.
\end{split}
\end{align}
Noting that Saloff-Coste's Sobolev inequality implies
$$e^{-c_0(1+\sqrt{\kappa}R)}V^{\frac{2}{n}}R^{-2}\left\|f^{\frac{\alpha_0+t-1}{2}+\frac{p}{4} }\eta\right\|_{L^{\frac{2n}{n-2}}(\Omega)}^2\leq \int_{\Omega}\left|\nabla \left(f^{\frac{\alpha_0+t-1}{2}+\frac{p}{4} }\eta\right)\right|^2+R^{-2}\int_\Omega f^{ \alpha_0+t +\frac{p}{2} -1 }\eta ^2,$$
furthermore, we can infer from \eqref{2.33} that
\begin{align}\label{3.32}
\begin{split}
& \beta_{n,p,q,r}\int_\Omega f^{\alpha_0+\frac{p}{2}+t}\eta^2 +
\frac{a_3}{t }e^{-c_0(1+\sqrt{\kappa}R)}V^{\frac{2}{n}}R^{-2}\left\|f^{\frac{\alpha_0+t-1}{2}+\frac{p}{4} }\eta\right\|_{L^{\frac{2n}{n-2}}}^2\\
\leq  &\ 2(n-1)\kappa  \int_\Omega f^{\alpha_0+t+\frac{p}{2}-1}\eta^2+\frac{a_4}{ t }\int_\Omega f^{\alpha_0+t+\frac{p}{2}-1}|\nabla\eta|^2
+\frac{a_3}{ t }\int_\Omega R^{-2}f^{ \alpha_0 +\frac{p}{2}+t-1 }\eta ^2.
\end{split}
\end{align}
		
Now we choose
\begin{align}\label{equa2.8}
t_0 = c_{n,p,q,r}\left(1+\sqrt{\kappa} R\right) \quad\text{and}\quad c_{n,p,q,r}=\max\left\{c_0+1, \frac{a_1^2}{a_2\beta_{n,p,q,r}} \right\},
\end{align}
and pick $t$ such that $t\geq t_0$. Since
\begin{align*}
2(n-1)\kappa  R^2\leq\frac{ 2(n-1)}{c_{n,p,q,r}^2}t_0^2\quad \text{ and}\quad \frac{a_3}{t}\leq \frac{a_3}{c_{n,p,q,r}},
\end{align*}
there exists $a_5 = a_5(n,p,q,r)>0$ such that
\begin{align}\label{a5}
2(n-1)\kappa  R^2+\frac{a_3}{t}\leq a_5t_0^2 = a_5c^2_{n,p,q,r}\left(1+\sqrt{\kappa} R\right)^2.
\end{align}
Immediately, it follows from \eqref{3.32} and \eqref{a5} that
\begin{align}\label{2.34}
\begin{split}
& \beta_{n,p,q,r}\int_\Omega f^{\alpha_0+\frac{p}{2}+t}\eta^2 + \frac{a_3}{ t }e^{-t_0}V^{\frac{2}{n}}R^{-2}\left\|f^{\frac{\alpha_0+t-1}{2}+\frac{p}{4} }\eta\right\|_{L^{\frac{2n}{n-2}}}^2\\
\leq\  & a_5t_0^2R^{-2} \int_\Omega f^{\alpha_0+\frac{p}{2}+t-1}\eta^2+\frac{a_4}{t }\int_\Omega f^{\alpha_0+\frac{p}{2}+t-1}|\nabla\eta|^2.
\end{split}
\end{align}
Thus, we finish the proof.
\end{proof}

\subsection{\texorpdfstring{$L^{\beta}$}\, bound of gradient in a geodesic ball with radius $3R/4$}\label{sect3.2}
	
\begin{lem}\label{lpbound}
Let $(M^n,g)$ be an $n$-dimensional complete Riemannian manifold with Ricci curvature $\ric_g\geq -(n-1)\kappa g$ where $\kappa$ is a non-negative constant and $\Omega = B_R(o)\subset M$ is a geodesic ball. Assume that $f$ is the same as in the above lemma. For a given positive number $$\beta = \left(\alpha_0+t_0+\frac{p}{2}-1\right)\frac{n}{n-2},$$
there exists a constant $a_8 = a_8(n,p, q)>0$ such that
\begin{align}\label{lpbpund}
\|f \|_{L^{\beta}(B_{3R/4}(o))}\leq a_8V^{\frac{1}{\beta}} \frac{t_0^2}{ R^2},
\end{align}
where $V$ is the volume of geodesic ball $B_R(o)$.
\end{lem}
	
\begin{proof}
We set $t=t_0$ in \eqref{2.34} and obtain
\begin{align}\label{equation:3.31}
\begin{split}
& \beta_{n,p,q,r}\int_\Omega f^{\alpha_0+\frac{p}{2}+t_0}\eta^2 + \frac{a_3}{ t_0 } e^{-t_0}V^{\frac{2}{n}}R^{-2}\left\|f^{\frac{\alpha_0+t_0-1}{2}+\frac{p}{4}}\eta\right\|_{L^{\frac{2n}{n-2}}}^2\\
\leq  &\ a_5t_0^2R^{-2} \int_\Omega f^{\alpha_0+\frac{p}{2}+t_0-1}\eta^2+\frac{a_4}{t_0 }\int_\Omega f^{\alpha_0+\frac{p}{2}+t_0-1}|\nabla\eta|^2.
\end{split}
\end{align}
From (\ref{equation:3.31}) we know that, if
		$$
		f\geq  \frac{2a_{5}t_0^2}{\beta_{n,p,q,r}R^2},
		$$
then there holds true
		$$
		a_5t_0^2R^{-2} \int_\Omega f^{\alpha+\frac{p}{2}+t_0-1}\eta^2\leq\frac{\beta_{n,p,q,r}}{2}\int_\Omega f^{\alpha_0+\frac{p}{2}+t_0}\eta^2.
		$$
		
Now we denote
$$\Omega_1 = \left\{f<  \frac{2a_{5}t_0^2}{\beta_{n,p,q,r}R^2} \right\},$$
then $\Omega$ can be decomposed into two parts $\Omega_1$ and $\Omega_2$, i.e., $\Omega=\Omega_1\cup\Omega_2$ where $\Omega_2$ is the complement of $\Omega_1$. Hence, we have
		\begin{align}\label{2.36}
			\begin{split}
				a_5t_0^2R^{-2} \int_\Omega f^{\alpha_0+\frac{p}{2}+t_0-1}\eta^2
				=&\ a_5t_0^2R^{-2} \int_{\Omega_1} f^{\alpha_0+\frac{p}{2}+t_0-1}\eta^2
				+a_5t_0^2R^{-2} \int_{\Omega_2} f^{\alpha_0+\frac{p}{2}+t_0-1}\eta^2
				\\
				\leq&\
				\frac{a_5t_0^2}{R^2} \left(\frac{2a_{5}t_0^2}{\beta_{n,p,q,r}R^2}\right)^{\alpha_0+\frac{p}{2} +t_0-1 }V
				+
				\frac{\beta_{n,p,q,r} }{2}\int_\Omega f^{\alpha_0+\frac{p}{2}+t_0}\eta^2,
			\end{split}
		\end{align}
where $V$ is the volume of $B_R(o)$. It follows from (\ref{2.34}) and (\ref{2.36})
		\begin{align}
			\label{2.37}
			\begin{split}
				&\frac{\beta_{n,p,q,r}}{2}\int_\Omega f^{\alpha_0+\frac{p}{2}+t_0}\eta^2
				+
				\frac{a_3}{ t_0 }e^{-t_0}V^{\frac{2}{n}}R^{-2}\left\|f^{\frac{\alpha_0+t_0-1}{2}+\frac{p}{4} }\eta\right\|_{L^{\frac{2n}{n-2}}}^2\\
				\leq \ &
				\frac{a_5t_0^2}{R^2} \left(\frac{2a_{5}t_0^2}{\beta_{n,p,q,r}R^2}\right)^{\alpha_0+\frac{p}{2} +t_0-1 }V
				+\frac{a_4}{ t_0 }\int_\Omega f^{\alpha_0+\frac{p}{2}+t_0-1}|\nabla\eta|^2.
			\end{split}
		\end{align}

Now we choose $\gamma\in C^{\infty}_0(B_R(o))$ satisfying
		$$
		\begin{cases}
			0\leq\gamma(x)\leq 1,\quad |\nabla\gamma(x)|\leq\frac{C }{R}, &\forall x\in B_R(o);\\
			\gamma(x)\equiv 1, & \forall x\in B_{ {3R}/{4}}(o),
		\end{cases}
		$$
and let $\eta = \gamma^{ \alpha_0 + \frac{p}{2}+t_0}$. Then, we have
		\begin{align}\label{314}
			a_4R^2 |\nabla\eta|^2
			\leq
			a_4 C^2 \left( t_0+ \frac{p}{2}+\alpha_0\right )^2\eta ^{\frac{2\alpha_0+2t_0+p-2}{\alpha_0+p/2+t_0}}
			\leq
			a_{6}t^2_0\eta^{\frac{2\alpha_0+2t_0+p-2}{\alpha_0+p/2+t_0}}.
		\end{align}
By H\"older inequality and Young inequality, we can deduce the following
\begin{align}\label{2.39}
\begin{split}
\frac{a_4}{t_0}\int_{\Omega}f^{\frac{p}{2}+\alpha_0+t_0-1}|\nabla\eta|^2\leq\  &\frac{a_6t_0}{R^2} \int_{\Omega}f^{\frac{p}{2}+\alpha_0+t_0-1}
\eta^{\frac{2\alpha_0+p+2t_0-2}{\alpha_0+p/2+t_0}}\\
\leq\ & \frac{a_6t_0}{R^2}  \left(\int_{\Omega}f^{\alpha_0+t_0+ \frac{p}{2} }\eta^2\right)^{\frac{\alpha_0+p/2+t_0-1}{\alpha+p/2+t_0}}V^{\frac{ 1}{\alpha_0+t_0+ p/2 }}\\
\leq \ & \frac{\beta_{n,p,q,r}}{2}\left[\int_{\Omega}f^{ \alpha_0+t_0+\frac{p}{2} }\eta^2 + \left(\frac{2a_{6}t_0 }{\beta_{n,p,q,r}R^2}\right)^{ \alpha_0+t_0+\frac{p}{2} }V\right].
\end{split}
\end{align}
Hence, we conclude from (\ref{2.37}) and (\ref{2.39}) that
\begin{align}\label{2.40}
\begin{split}
&\left(\int_{\Omega}f^{\frac{n(p/2+\alpha_0+t_0-1)}{n-2}}\eta^{\frac{2n}{n-2}}\right)^{\frac{n-2}{n}}\\
\leq\ &\frac{t_0}{a_3} e^{t_0}V^{1-\frac{2}{n}}R^2\left[\frac{2a_5t_0^2}{R^2} \left(\frac{2a_{5}t_0^2}{\beta_{n,p,q,r}R^2}\right)^{\alpha_0+t_0+\frac{p}{2}-1 } + \frac{a_{6}t^2_0}{R^2} \left(\frac{2a_{6}t_0 }{\beta_{n,p,q,r}R^2}\right)^{\alpha_0+t_0+\frac{p}{2} -1 }\right]\\
\leq\ &a_7^{t_0}e^{t_0}V^{1-\frac{2}{n}}t_0^3\left( \frac{t_0^2}{ R^2}\right)^{\alpha_0+t_0+\frac{p}{2} -1},
\end{split}
\end{align}
where the constant $a_7$ depends only on $n,\,p,\, q$ and $r$, and satisfies
$$
a_7^{t_0} \geq \frac{2a_5}{a_3}\left(\frac{4a_5}{\beta_{n,p,q,r}}\right)^{\alpha_0+t_0+\frac{p}{2}-1}
+\frac{a_6}{a_3}\left(\frac{4a_6}{\beta_{n,p,q,r}t_0}\right)^{\alpha_0+t_0+\frac{p}{2}-1}.
$$
Here we have used the fact $t_0\geq1$.
		
Taking $\frac{1}{\alpha_0+t_0+p/2-1}$ power of the both sides of (\ref{2.40}) gives
\begin{align}
\left\|f\eta^{\frac{2}{\alpha_0+t_0+p/2-1}}\right\|_{L^{\beta}(\Omega)}\leq a_7^{\frac{t_0}{\alpha_0+t_0+p/2-1}}V^{1/\beta}
t_0^{\frac{3}{\alpha_0+t_0+p/2-1}}\frac{t_0^2}{ R^2}\leq a_8V^{\frac{1}{\beta}} \frac{t_0^2}{ R^2},
\end{align}
where $a_8$ depends only on $n,\ p,\ q$ and $r$, and satisfies
		$$
		a_8 \geq a_7^{\frac{t_0}{\alpha_0+t_0+p/2-1}}t_0^{\frac{3}{\alpha_0+t_0+p/2-1}}.
		$$
Since $\eta\equiv1$ in $B_{3R/4}$, we obtain that
		$$
		\|f \|_{L^{\beta}(B_{3R/4}(o))}\leq a_8V^{\frac{1}{\beta}} \frac{t_0^2}{ R^2}.
		$$
Thus, the proof of this lemma is finished.
\end{proof}

\subsection{Moser iteration}\label{sec3.3}\

Now we turn to providing the proof of \thmref{thm1}.
	
\begin{proof}
We discard the first term in (\ref{2.34}) to obtain
\begin{align}\label{2.42}
\frac{a_3}{ t }e^{-t_0}V^{\frac{2}{n}}R^{-2}\left\|f^{\frac{\alpha_0+t-1}{2}+\frac{p}{4} }\eta\right\|_{L^{\frac{2n}{n-2}}}^2\leq
\frac{a_5\alpha_0^2}{R^2}  \int_\Omega f^{\alpha_0+\frac{p}{2}+t-1}\eta^2+\frac{a_4}{ t }\int_\Omega f^{\alpha_0+\frac{p}{2}+t-1}|\nabla\eta|^2.
\end{align}
For $k=1,\, 2,\, \cdots$, let $$r_k = \frac{R}{2}+\frac{R}{4^k}$$ and $\Omega_k = B_{r_k}(o)$. Choose a sequence of cut-off function $\eta_k\in C^{\infty}(\Omega_k)$ which satisfy
\begin{align}
\begin{cases}
0\leq \eta_k(x)\leq1, \quad |\nabla\eta_k(x)|\leq \frac{C4^k}{R}, & \forall x\in B_{r_{k+1}}(o);\\
\eta_k(x)\equiv 1, &\forall x\in B_{r_{k+1}}(o),
\end{cases}
\end{align}
where $C$ is a constant that does not depend on $(M,g)$ and equation \eqref{equa1}. Substituting $\eta=\eta_k$ into  \eqref{2.42}, we arrive at
\begin{align*}
a_3e^{-t_0}V^{\frac{2}{n}} \left\|f^{\frac{\alpha_0+t-1}{2}+\frac{p}{4} }\eta_k\right\|_{L^{\frac{2n}{n-2}}(\Omega_k)}^2\leq \ &
a_5t_0^2t\int_{\Omega_k} f^{\alpha_0+\frac{p}{2}+t-1}\eta_k^2+ a_4R^2\int_{\Omega_k} f^{\alpha_0+\frac{p}{2}+t-1}\left|\nabla\eta_k\right|^2\\
\leq\ & \left(a_5t_0^2t + C^216^k\right)\int_{\Omega_k} f^{\alpha_0+\frac{p}{2}+t-1}.
\end{align*}
		
Now we choose $\beta_1=\beta$, $\beta_{k+1}=n\beta_k/(n-2)$, $k=1,\,2,\, \cdots$, and let $t=t_k$ such that
$$t_k+\frac{p}{2}+\alpha_0-1=\beta_k,$$
then we have
\begin{align*}
a_3 \left(\int_{\Omega_k}f^{\beta_{k+1}}\eta_k^\frac{2n}{n-2}\right)^{\frac{n-2}{n}} \leq\
e^{t_0}V^{-\frac{2}{n}}\left(a_5t_0^2\left(t_0+\alpha_0+\frac{p}{2}-1\right)\left(\frac{n}{n-2}\right)^k + C^216^k\right)\int_{\Omega_k} f^{\beta_k}.
\end{align*}
Since $n/(n-2)<16, ~\forall n>2$, there exists some constant $a_9 $ such that
\begin{align}\label{eq:2.45}
\left(\int_{\Omega_k}f^{\beta_{k+1}}\eta_k^\frac{2n}{n-2}\right)^{\frac{n-2}{n}} \leq \ a_9t_0^316^k e^{ t_0}V^{-\frac{2}{n}}  \int_{\Omega_k} f^{\beta_k}.
\end{align}
Taking $1/\beta_k$ power of the both sides of (\ref{eq:2.45}) respectively and using the fact
$$\eta_k\equiv1\in\Omega_{k+1},$$
we obtain
\begin{align}\label{2.46}
\begin{split}
\|f\|_{L^{\beta_{k+1}}(\Omega_{k+1})}\leq & \left(a_9t_0^316^k e^{ t_0}V^{-\frac{2}{n}}\right)^{\frac{1}{\beta_k}} 16 ^{\frac{k}{\beta_k}}\|f\|_{L^{\beta_k}(\Omega_k)}.
\end{split}
\end{align}
Noting
$$\sum_{k=1}^{\infty}\frac{1}{\beta_k} =\frac{n}{2\beta_1} \quad\text{and}\quad \sum_{k=1}^{\infty}\frac{k}{\beta_k}=\frac{n^2}{4\beta_1},$$
we have
\begin{align}\label{2.47}\|f\|_{L^{\infty}(B_{R/2}(o))}\leq \
a_{10}  V^{-\frac{1}{\beta}} \|f\|_{L^{\beta_1}(B_{3R/4}(o))},
\end{align}
where
		$$
		a_{10} \geq \left(a_9t_0^316^k e^{ t_0}\right)^{\frac{n}{2\beta_1 }} 16 ^{\frac{n^2}{4\beta_1}}.
		$$
By (\ref{lpbpund}), we obtain
		\begin{align}\label{3.44}
			\|f\|_{L^{\infty}(B_{R/2}(o))}\leq\ a_{11}\frac{(1+\sqrt{\kappa}R)^2}{R^2},
		\end{align}
where $a_{11} = a_{10}a_8c_{n,p,q,r}$. Reminding $f=|\nabla u|^2$ and $u=-(p-1)\log v$, we obtain the desired estimate. Thus, we complete the proof of \thmref{thm1}.
\end{proof}

\subsection{Proof of \thmref{thm1.4} }
In this case, we have $\kappa=0$. By Theorem 1.1, we have
\begin{align}\label{3.45}
\frac{|\nabla v(x_0)|}{v(x_0)}\leq \sup_{B(x_0,\frac{R}{2})}\frac{|\nabla v|}{v}\leq \frac{C(n,p,q,r)}{R}, \quad \forall x_0\in M.
\end{align}
Letting $R\to\infty$ in \eqref{3.45}, we obtain
	$$
	\nabla v(x_0)=0,\quad\forall x_0\in M.
	$$
Therefore, $v$ is a constant and $\Delta_pv=0$. If $bc>0$, then \eqref{equa1} admits no positive solutions; if $b>0,c<0$ and $bv^q+cv^r\not\equiv 0$, then the unique solution of \eqref{equa0} is $v=(-\frac{c}{b})^{\frac{1}{q-r}}$. If $q=r$, or $b=0, c\neq 0$, or $b\neq0, c=0$, the equation \eqref{equa1} is reduced to Lane-Emden equation and the related Liouville theorem is known in \cite{hewangwei2024}.
\qed
	
\subsection{Proof of \thmref{thm1.5}}\

For any $x\in B_{R/2}(o)\subset M$,  by Theorem 1.1, we have
	\begin{align}
		\label{3.48}
		|\nabla \ln v(x)|\leq C(n,p,q,r)\frac{ 1+\sqrt{\kappa}R }{R }.
	\end{align}
Choosing a minimizing geodesic $\gamma(t)$ with arc length parameter connecting $o$ and $x$:
	$$
	\gamma:[0,d]\to M,\quad\gamma(0)=o, \quad \gamma(d)=x.
	$$
where $d=d(x,o)$ is the geodesic distance from $o$ to $x$, we have
	\begin{align}
		\ln v(x)-\ln v(o)=\int_0^d\frac{d}{dt}\ln v\circ\gamma(t)dt.
	\end{align}
Since $|\gamma'|=1$ , we have
	\begin{align}
		\left|\frac{d}{dt}\ln v\circ\gamma(t)\right|\leq |\nabla \ln v||\gamma'(t)| \leq C(n,p,q,r)\frac{ 1+\sqrt{\kappa}R }{R }.
	\end{align}	
Thus it follows from $d\leq R/2$ and the above inequality that
	$$
	e^{-C(n,p,q,r)(1+\sqrt{\kappa}R)/2}\leq v(x)/v(o)\leq e^{C(n,p,q,r)(1+\sqrt{\kappa}R)/2}.
	$$
So, for any $x, y\in B_{R/2}(o)$ we have
	$$
	v(x)/v(y)\leq e^{C(n,p,q,r)(1+\sqrt{\kappa}R)}.
	$$

Now suppose $v$ is a global positive solution of equation \eqref{equa1} on $M$. Letting $R\to\infty$ in \eqref{3.48}, we obtain that
	$$
	|\nabla \ln v(x)|\leq C(n,p,q,r)\sqrt{\kappa}, \quad \forall x\in M.
	$$
For any $y\in M$, choose a minimizing geodesic $\gamma(t)$ parameterized by arc length which connects $x$ and $y$:
	$$
	\gamma:[0,d]\to M,\quad\gamma(0)=x, \quad \gamma(d)=y.
	$$
where $d=d(x,y)$ is the distance from $x $ to $y$. Then, we have
	\begin{align}\label{3.49}
		\ln v(y)-\ln v(x)=\int_0^d\frac{d}{dt}\ln v\circ\gamma(t)dt.
	\end{align}
Since $|\gamma'(t)|=1$, we have
	\begin{align}\label{3.50}
		\left|\frac{d}{dt}\ln v\circ\gamma(t)\right|\leq |\nabla \ln v||\gamma'(t)| = C(n,p,q,r)\sqrt{\kappa}.
	\end{align}
It follows from (\ref{3.49}) and (\ref{3.50}) that
	\begin{align}\label{3.51}
		v(y)/v(x)\leq e^{C(n,p,q,r)\sqrt{\kappa}d(x, y)}.
	\end{align}
Thus we finish the proof by (\ref{3.51}).

\subsection{The case $\dim(M)=2$.}\

In the proof of \thmref{thm1} above, we have used Saloff-Coste's Sobolev inequality on the embedding $W^{1,2}(B)\hookrightarrow
L^{\frac{2n}{n-2}}(B)$ on a manifold. Since the Sobolev exponent $2^*=2n/(n-2)$ requires $n>2$, we did not consider the case $n=2$ in \thmref{thm1}. In this subsection, we will explain briefly that \thmref{thm1} can also be established when $n=2$.
	
When $\dim M=n=2$, we need the special case of  Saloff-Coste's Sobolev theorem, i.e., \lemref{salof}. For any $n'> 2$, there holds
	$$
	\|f\|_{L^{\frac{2n'}{n'-2}}(B)}^2\leq e^{c_0(1+\sqrt{\kappa}R)}V^{-\frac{2}{n'}}R^2\left(\int_B|\nabla f|^2+R^{-2}f^2\right)
	$$for any $f\in C^{\infty}_0(B)$.
	
For example, we can choose $n'=4$, then we have	
\begin{align*}
\left(\int_{\Omega}f^{2\alpha+p}\eta^{4}\right)^{\frac{1}{2}}\leq e^{c_0(1+\sqrt{\kappa} R)}V^{-\frac{1}{2}}\left(R^2\int_{\Omega}\left|\nabla \left(f^{\frac{p}{4}+\frac{\alpha}{2}}\eta\right)\right|^2+\int_{\Omega}f^{\frac{p}{2}+\alpha}\eta^2\right).
\end{align*}
By the above inequality and \eqref{equation:3.31}, we can deduce the following integral inequality by almost the same method as in \lemref{lem3.2}.
\begin{align} \label{equation3.19}
\begin{split}
			& \beta_{n,p,q,r}\int_\Omega f^{\alpha_0+\frac{p}{2}+t}\eta^2 +
			\frac{a_3}{ t }e^{-t_0}V^{\frac{1}{2}}R^{-2}\left\|f^{\frac{\alpha_0+t-1}{2}+\frac{p}{4} }\eta\right\|_{L^{4}}^2\\
			\leq\  & a_5t_0^2R^{-2} \int_\Omega f^{\alpha_0+\frac{p}{2}+t-1}\eta^2+\frac{a_4}{t }\int_\Omega f^{\alpha_0+\frac{p}{2}+t-1}|\nabla\eta|^2,
\end{split}
\end{align}
where $\alpha_0$ is the same as  $\alpha_0$ defined in Section \ref{sect:3.1}, but the constants $a_i$ ($i=3, \, 4,\, 5$) may differ from those defined in Section \ref{sect:3.1}.
	
By repeating the same procedure as in Section \ref{sect3.2}, we can deduce from \eqref{equation3.19} the $L^{\beta}$-bound of $f $ in a geodesic ball with radius $3R/4$
\begin{align}\label{equation:3.20}
\|f \|_{L^{\beta}(B_{3R/4}(o))}\leq a_8V^{\frac{1}{\beta}} \frac{t_0^2}{ R^2},
\end{align}
where $\beta=p+2\alpha_0+2t_0-2$.
	
For the Nash-Moser iteration, we set $$\beta_l = 2^l(\alpha_0+t_0+p/2-1)$$ and $\Omega_l$ by the similar way with that in Section \ref{sec3.3}, and can obtain the following inequality
\begin{align}\label{equation:3.21}
\|f\|_{L^{\infty}(B_{\frac{R}{2}}(o))}\leq\ &a_{11} V^{-\frac{1}{\beta}}\|f\|_{L^{\beta}(B_{\frac{3R}{4}}(o))}.
\end{align}
Combining \eqref{equation:3.20} and \eqref{equation:3.21}, we finally obtain the Cheng-Yau type gradient estimate. Harnack inequality and Liouville type results follow from the Cheng-Yau type gradient estimate, here we omit the details.

\section{Global gradient estimate}
By the local gradient of positive solution $v$ of the equation \eqref{equa1}, we have
\begin{align*}
\sup_{B_{R/2}(o)}\frac{|\nabla v|}{v}\leq C(n,p,q,r)\frac{1+\sqrt{\kappa}R}{R}.
\end{align*}
If $v$ is a global solution, letting $R\to\infty$ in the above local gradient estimate, we obtain
\begin{align*}
\frac{|\nabla v|}{v}(x)\leq C(n,p,q,r)\sqrt{\kappa}, \quad\forall x\in M.
\end{align*}
In this section, we give a explicit expression of the above $C(n,p,q,r)$ in this section. Our approach is based on an ingenious idea from Sung and Wang(\cite{MR3275651}).

\subsection{Explicit upper bounds of $C(n,p,q,r)$: Proof of \thmref{t10}}\

\begin{thm}[\thmref{t10}]
Let $(M^n,g)$ be a complete non-compact Riemannian manifold with $\mathrm{Ric}_g\geq-(n-1)\kappa$ where $\kappa$ is a non-negative constant. Let $v$ be a global positive solution of \eqref{equa1} in $M$.
\begin{enumerate}
\item If
	\begin{align}\label{61}
		b\left(\frac{n+1}{n-1}-\frac{q}{p-1}\right)\geq 0 \quad \text{and}\quad c\left(\frac{n+1}{n-1}-\frac{r}{p-1}\right)\geq 0,
	\end{align}
then
	\begin{align*}
		\frac{|\nabla v|}{v}(x)\leq \frac{(n-1)\sqrt{\kappa}}{p-1}, \quad\forall x\in M.
	\end{align*}
	
\item If
	\begin{align}\label{62}
		c\leq 0\quad\text{and}\quad \left|\frac{q}{p-1}-\frac{n+1}{n-1}\right|<\frac{2}{n-1},
	\end{align}
then
\begin{align*}
\frac{|\nabla v|}{v}(x)\leq \frac{2\sqrt{\kappa}}{(p-1)\sqrt{\left(\frac{q}{p-1}-1\right)\left(\frac{n+3}{n-1}-\frac{q}{p-1}\right)}}, \quad\forall x\in M.
\end{align*}
\item If
\begin{align}\label{63}
b\geq 0\quad \text{and}\quad \left|\frac{r}{p-1}-\frac{n+1}{n-1}\right|< \frac{2}{n-1},
\end{align}
then
\begin{align*}
\frac{|\nabla v|}{v}(x)\leq \frac{2\sqrt{\kappa}}{(p-1)\sqrt{\left(\frac{r}{p-1}-1\right)\left(\frac{n+3}{n-1}-\frac{r}{p-1}\right)}}, \quad\forall x\in M.
\end{align*}
\end{enumerate}
\end{thm}

In order to prove the above theorem (\thmref{t10}), we need to establish the following two lemmas.
\begin{lem}\label{lem42}
Let $(M,g)$ be an $n$-dim($n\geq 2$) complete non-compact Riemannian manifold with $\mathrm{Ric}_g\geq-(n-1)\kappa$ where $\kappa$ is a non-negative constant. Let $v$ be a global positive solution of \eqref{equa1} in $M$.
\begin{enumerate}
\item If \eqref{61} is satisfied, equivalently $(b,c,q,r)\in W_1$,
we denote $y_1=(n-1)^2\kappa$. For any $\delta>0$ and $\omega = (f-y_1-\delta)^+$, we have
\begin{align*}
\mathcal{L}(\omega^\alpha)\geq 	2\alpha\omega^{\alpha-1}  \left(k_1\omega-k_2|\nabla \omega|\right),
\end{align*}
where $\alpha>2$ and $k_1, k_2$ are two positive constants depending on $n,p,q,r,\kappa$ and $\delta$.
\item If \eqref{62} is satisfied, i.e. $(b,c,q,r)\in W_2$,
we denote
\begin{align*}
y_2 = \frac{4\kappa}{\left(\frac{n+3}{n-1}-\frac{q}{p-1}\right)\left( \frac{q}{p-1}-1\right)}.
\end{align*}
For any $\delta>0$, $\omega = (f-y_2-\delta)^+$ and $\alpha$ large enough, we have
\begin{align*}
\mathcal{L}(\omega^\alpha)\geq 	2\alpha\omega^{\alpha-1}  \left(l_1\omega-l_2|\nabla \omega|\right),
\end{align*}
where $l_1, l_2$ are two positive constants depending on $n,p,q,r,\kappa$ and $\delta$.
\item If \eqref{63} is satisfied, i.e. $(b,c,q,r)\in W_3$, we denote
\begin{align*}
y_3 = \frac{4\kappa}{\left(\frac{n+3}{n-1}-\frac{r}{p-1}\right)\left( \frac{r}{p-1}-1\right)}.
\end{align*}
For any $\delta>0$, $\omega = (f-y_3-\delta)^+$ and $\alpha$ large enough, we have
\begin{align*}
\mathcal{L}(\omega^\alpha)\geq 	2\alpha\omega^{\alpha-1}  \left(l_1\omega-l_2|\nabla \omega|\right),
\end{align*}
where $l_1, l_2$ are two positive constants depending on $n,p,q,r,\kappa$ and $\delta$.
\end{enumerate}
\end{lem}

\begin{proof}
\textbf{Case 1.}
According to \lemref{lem3.1}, if \eqref{61} is satisfied, there holds
\begin{align}\label{eq412}
\frac{f^{2-\alpha-\frac{p}{2}}}{2\alpha}\mathcal{L} (f^{\alpha})\geq \frac{f^2}{n-1}-(n-1)\kappa f -\frac{a_1}{2} f^{\frac{1}{2}}|\nabla f|.
\end{align}
On $\{f\geq y_1+\delta\}$, the definition of operator $\mathcal{L}$ yields that
\begin{align*}
\mathcal{L}(\omega^\alpha) = \di(\alpha\omega^{\alpha-1}A(\nabla f))=\alpha(\alpha-1)\omega^{\alpha-2}\langle\nabla\omega, A(\nabla f)\rangle+\alpha\omega^{\alpha-1}\mathcal{L}(f).
\end{align*}
In the above expression, it is easy to see that $\nabla\omega=\nabla f$ in $\{f>y_1+\delta\}$, but $\nabla\omega$ causes a distribution on $\{f=y_1+\delta\}$. Now, if we assume $\alpha>2$, then the distribution caused by $\nabla\omega$ on $\{f=y_1+\delta\}$ is eliminated by $\omega$ since $\omega=0$  on $\{f=y_1+\delta\}$.

Since $f\geq \omega$ and $\langle\nabla f, A(\nabla f)\rangle\geq 0$, it follows that
\begin{align*}
\mathcal{L}(\omega^\alpha)\geq &\ \omega^{\alpha-1}\left(\alpha(\alpha-1)f^{-1}\langle\nabla f, A(\nabla f)\rangle+\alpha \mathcal{L}(f)\right)=\ \frac{\omega^{\alpha-1}}{f^{\alpha-1}}\mathcal{L}(f^\alpha).
\end{align*}
It follows from \eqref{eq412} and the above inequality that
\begin{align*}
\mathcal{L} (f^{\alpha})\geq &\ 2\alpha\omega^{\alpha-1}f^{\frac{p}{2}-1}  \left(\frac{f^2}{n-1}-(n-1)\kappa f-\frac{a_1}{2}f^{\frac{1}{2}}|\nabla \omega|\right).
\end{align*}
Since $f\geq y_1+\delta = (n-1)^2\kappa+\delta$, The above inequality can be rewritten as
\begin{align*}
\mathcal{L} (f^{\alpha})\geq &\ 2\alpha\omega^{\alpha-1}f^{\frac{p-1}{2}}  \left(f^{\frac{1}{2}}\left(\frac{f}{n-1}-(n-1)\kappa \right)-\frac{a_1}{2}|\nabla \omega|\right)\\
\geq& \ 2\alpha\omega^{\alpha-1}f^{\frac{p-1}{2}}  \left(f^{\frac{1}{2}} \frac{\delta}{n-1} -\frac{a_1}{2}|\nabla \omega|\right).
\end{align*}
Using the fact $$y_1+\delta\leq f\leq C(n,p,q,r)\sqrt{\kappa}$$ on $\{f\geq y_1+\delta\}$, we obtain
\begin{align*}
\mathcal{L}(f^\alpha)\geq 	2\alpha\omega^{\alpha-1}  \left(k_1\omega-k_2|\nabla \omega|\right), \quad\forall \alpha >2,
\end{align*}
where $k_1$ and $k_2$ are two positive constants depending on $n,p,q,r,\kappa, \delta$.
	
\textbf{Case 2.}
If \eqref{62} is satisfied, the inequality \eqref{equa:3.14a} implies that
\begin{align*}
\frac{f^{2-\alpha-\frac{p}{2}}}{2\alpha}\mathcal{L} (f^{\alpha})\geq &\ \beta_{n,p,q,\alpha}f^2-(n-1)\kappa f -\frac{a_1}{2}f^{\frac{1}{2}}|\nabla f|,
\end{align*}
where
$$
\beta^{(2)}_{n,p,q } =\lim_{\alpha\to\infty}\beta_{n,p,q, \alpha}= \frac{1}{n-1}-\frac{n-1}{4}\left(\frac{q}{p-1}-\frac{n+1}{n-1}\right)^2>0.
$$
Then, on $\{f\geq y_2+\delta\}$,
	$$
	f\geq y_2+\delta=\frac{ \kappa(n-1)}{\beta^{(2)}_{n,p,q }}+\delta.
	$$
Direct calculation shows that
\begin{align*}
\mathcal L(\omega^\alpha) \geq & \ 2\alpha\omega^{\alpha-1}f^{\frac{p-1}{2}}  \left(f^{\frac{1}{2}}\left(\beta_{n, p, q,\alpha} f  -(n-1)\kappa \right)-\frac{a_1}{2}|\nabla \omega|\right).
\end{align*}
Since
	$$
	\lim_{\alpha\to\infty}\beta_{n, p, q, \alpha} f  -(n-1)\kappa=\beta^{(2)}_{n,p,q }f  -(n-1)\kappa\geq \beta^{(2)}_{n,p,q }\delta,
	$$
we can choose $\alpha$ large enough such that
$$
\beta_{n, p, q, \alpha} f  -(n-1)\kappa>\frac{1}{2}\beta^{(2)}_{n,p,q}\delta.
$$
It follows
\begin{align*}
\mathcal L(\omega^\alpha) \geq & \ 2\alpha\omega^{\alpha-1}f^{\frac{p-1}{2}}  \left(f^{\frac{1}{2}}\frac{\beta^{(2)}_{n, p, q} \delta}{2}-\frac{a_1}{2}|\nabla \omega|\right)\\
=& 2\alpha\omega^{\alpha-1} \left(f^{\frac{p}{2}} \frac{\beta^{(2)}_{n,p,q }\delta}{2}-\frac{a_1}{2}f^{\frac{p-1}{2}}|\nabla \omega|\right).
\end{align*}
Using the fact $$y_2+\delta\leq f\leq C(n,p,q,r)\sqrt{\kappa}\quad \text{and}\quad f\geq\omega$$ on $\{f\geq y_2+\delta\}$, we have
\begin{align*}
\mathcal{L}(\omega^\alpha)\geq& 2\alpha\omega^{\alpha-1} \left(f^{\frac{p}{2}-1} \frac{\beta^{(2)}_{n,p,q }\delta}{2}\omega-\frac{a_1}{2}f^{\frac{p-1}{2}}|\nabla \omega|\right)\\
\geq &2\alpha\omega^{\alpha-1}(l_1\omega-l_2|\nabla\omega|),
\end{align*}
where $l_1, l_2$ are two positive constants depending on $n,p,q,r,\kappa, \delta$.

\textbf{Case 3.}
If \eqref{63} is satisfied, the inequality \eqref{equa:3.15a} implies that
\begin{align*}
\frac{f^{2-\alpha-\frac{p}{2}}}{2\alpha}\mathcal{L} (f^{\alpha})\geq &\ \beta_{n,p,r,\alpha}f^2-(n-1)\kappa f -\frac{a_1}{2}f^{\frac{1}{2}}|\nabla f|,
\end{align*}
where
	$$
	\beta^{(3)}_{n,p,r} =\lim_{\alpha\to\infty}\beta_{n,p,r, \alpha}= \frac{1}{n-1}-\frac{n-1}{4}\left(\frac{r}{p-1}-\frac{n+1}{n-1}\right)^2>0.
	$$
Then, on $\{f\geq y_3+\delta\}$,
	$$
	f\geq y_3+\delta=\frac{ \kappa(n-1)}{\beta^{(3)}_{n,p,r}}+\delta.
	$$
Direct calculation shows that
\begin{align*}
\mathcal L(\omega^\alpha) \geq & \ 2\alpha\omega^{\alpha-1}f^{\frac{p-1}{2}}  \left(f^{\frac{1}{2}}\left(\beta_{n, p, r,\alpha} f  -(n-1)\kappa \right)-\frac{a_1}{2}|\nabla \omega|\right).
\end{align*}
Since
	$$
	\lim_{\alpha\to\infty}\beta_{n, p, r, \alpha} f  -(n-1)\kappa=\beta^{(3)}_{n,p,r}f  -(n-1)\kappa\geq \beta^{(3)}_{n,p,r }\delta,
	$$
we can choose $\alpha$ large enough such that
	$$
	\beta_{n, p, r, \alpha} f  -(n-1)\kappa>\frac{1}{2}\beta^{(3)}_{n,p,r}\delta.
	$$
It follows
\begin{align*}
\mathcal L(\omega^\alpha) \geq & \ 2\alpha\omega^{\alpha-1}f^{\frac{p-1}{2}}  \left(f^{\frac{1}{2}}\frac{\beta^{(3)}_{n,p,r}\delta}{2}-\frac{a_1}{2}|\nabla \omega|\right)\\
=& 2\alpha\omega^{\alpha-1} \left(f^{\frac{p}{2}} \frac{\beta^{(3)}_{n,p,r }\delta}{2}-\frac{a_1}{2}f^{\frac{p-1}{2}}|\nabla \omega|\right).
\end{align*}
Using the facts
$$y_3+\delta\leq f\leq C(n,p,q,r)\sqrt{\kappa}\quad \text{and}\quad f\geq\omega$$
on $\{f\geq y_3+\delta\}$, we have
\begin{align*}
\mathcal{L}(\omega^\alpha)\geq& 2\alpha\omega^{\alpha-1} \left(f^{\frac{p}{2}-1} \frac{\beta^{(3)}_{n,p,r}\delta}{2}\omega-\frac{a_1}{2}f^{\frac{p-1}{2}}|\nabla \omega|\right)\\
\geq &2\alpha\omega^{\alpha-1}(l_1\omega-l_2|\nabla\omega|),
\end{align*}
where $l_1$ and $l_2$ are two positive constants depending on $n,\, p,\,q,\, r,\, \kappa$, and $\delta$.
\end{proof}

\begin{lem}\label{lem43}
Let $(M,g)$ be an $n$-dim($n\geq 2$) complete Riemannian manifold satisfying $\mathrm{Ric}_g\geq-(n-1)\kappa g$ for some constant $\kappa\geq0$. Let $v\in C^1(M)$ be an entire positive solution of \eqref{equa1}. For some $y>0$, we define $\omega = (f-y)^+$. If $\omega$ satisfies the following inequality
$$	 		
\mathcal L(\omega^\alpha)\geq \omega^{\alpha-1}(l_1\omega-l_2|\nabla\omega|),
$$
where $l_1$ and $l_2$ and $\alpha$ are some positive constants, then $\omega\equiv 0$, i.e.,  $f\leq y$.
\end{lem}

\begin{proof}
Let $\eta\in C^{\infty}_0(M,\mathbb R)$ be a cut-off function to be determine later. We choose $\eta^2\omega^\gamma$ as test function, then there holds
\begin{align*}
\int_M\mathcal L(\omega^\alpha)\omega^{\gamma}\eta^2\geq \int_M2\alpha\omega^{\alpha+\gamma-1}  \left(l_1\omega-l_2|\nabla \omega|\right)\eta^2.
\end{align*}
We omit the term $f^{\frac{p}{2}-1}$ since $f$ is uniform bounded on $M$. By integration by parts, we obtain
\begin{align*}
&\int_M2\alpha\omega^{\alpha+\gamma-1}  \left(l_1\omega-l_2|\nabla \omega|\right)\eta^2\\
\leq& -\int_M\alpha\gamma\omega^{\alpha+\gamma-2}\eta^2(|\nabla\omega|^2+(p-2)f^{-1}\langle\nabla\omega, \nabla u\rangle^2)\\
&-\int_M2\alpha\eta\omega^{\alpha+\gamma-1}\left(\langle\nabla\omega, \nabla \eta \rangle+(p-2)f^{-1}\langle\nabla\omega,\nabla u\rangle\langle \nabla u,\nabla\eta\rangle\right).
\end{align*}
It follows
\begin{align*}
&\int_M2 \omega^{\alpha+\gamma-1}  \left(l_1\omega-l_2|\nabla \omega|\right)\eta^2+a_2\int_M \gamma\omega^{\alpha+\gamma-2}\eta^2|\nabla\omega|^2\\
\leq &2(p+1)\int_M\omega^{\alpha+\gamma-1}|\nabla\omega||\nabla \eta|  \eta,
\end{align*}
where $a_2=\min\{1, p-1\}>0$ is defined in Section 2.

We note that Cauchy inequality implies
\begin{align*}
2l_2 \omega^{\alpha+\gamma-1}  |\nabla \omega| \eta^2\leq&\ l_2\eta^2\omega^{\alpha+\gamma-2}\left(\frac{|\nabla\omega|^2}{\epsilon}+\epsilon\omega^2\right);\\
2 \eta(p+1) \omega^{\alpha+\gamma-1}|\nabla\omega||\nabla \eta| \leq& \ (p+1)  \omega^{\alpha+\gamma-2}\left(\frac{\eta^2|\nabla\omega|^2}{\epsilon}+\epsilon|\nabla\eta|^2\omega^2\right).
\end{align*}
If we choose $\epsilon$ such that
\begin{align*}
\frac{l_2+p+1}{\epsilon}=a_2\gamma,
\end{align*}
then we have
\begin{align*}
&\int_M2l_1 \omega^{\alpha+\gamma} \eta^2 \leq \int_M  \epsilon(p+1)  \omega^{\alpha+\gamma }|\nabla\eta|^2 +\int_M\epsilon l_2
\omega^{\alpha+\gamma} \eta^2.
\end{align*}
Picking $\gamma$ large enough such that $\epsilon l_2<l_1$, we obtain
\begin{align*}
&l_1 \int_M \omega^{\alpha+\gamma}\eta^2 \leq \int_M \epsilon(p+1)\omega^{\alpha+\gamma }|\nabla\eta|^2.
\end{align*}

Now, if $\omega\neq 0$, without loss of generality we may assume that $\omega|_{B_1}\neq 0$ for some geodesic ball $B_1$ with radius 1, it follows that there always holds true
$$\int_{B_1}\omega^{\alpha+\gamma}>0.$$
By choosing $\eta\in C^{\infty}_0(B_{k+1})$, $\eta\equiv 1$ in $B_k$ and $|\nabla\eta|<4$, then from the above argument we have
\begin{align*}
l_1\int_{B_k}\omega^{\alpha+\gamma}\leq 4\epsilon(p+1) \int_{B_{k+1}}\omega^{\alpha+\gamma}.
\end{align*}
By iteration on $k$, we have
\begin{align*}
\int_{B_k}\omega^{\alpha+\gamma}\geq \left(\frac{C}{\epsilon}\right)^k\int_{B_1}\omega^{\alpha+\gamma}.
\end{align*}
Since $\omega$ is uniformly bounded and
$$\vol(B_{k})\leq e^{(n-1)k\sqrt{\kappa}}$$
by volume comparison theorem, we have
\begin{align*}
C_1^{\alpha+\gamma}e^{(n-1)k\sqrt{\kappa}}\geq e^{k\ln\frac{C}{\epsilon}}.
\end{align*}
We choose $\gamma$ such that $$\ln\frac{C}{\epsilon}>2(n-1)\sqrt{\kappa}+2$$ and then choose $k$ large enough, immediately we get a contradiction.
This contradiction means $\omega\equiv 0$ and we finish the proof of this lemma.
\end{proof}

\begin{proof}[Proof of \thmref{t10}]
\thmref{t10} is the direct consequence of \lemref{lem42} and \lemref{lem43}. If $(b, c, q, r)\in W_i, i=1,2,3$, then by \lemref{lem42}, for any $\delta>0$ and $\omega=(f-y_i-\delta)^+$, we have
	$$
	\mathcal{L}(\omega^\alpha)\geq 	2\alpha\omega^{\alpha-1}  \left(l_1\omega-l_2|\nabla \omega|\right).
	$$
Due to \lemref{lem43}, we have $\omega\equiv 0$, i.e., $f\leq y_i+\delta$. Since $\delta$ can be arbitrary small,  it is clear that $f\leq y_i$.
Keeping $$f = |\nabla u|^2 = (p-1)^2|(\nabla v)/v|^2$$
in mind, we obtain
	$$
	\frac{|\nabla v|}{v}\leq \frac{\sqrt{y_i}}{p-1}.
	$$
So we finish the proof of \thmref{t10}.
\end{proof}

Note $y_1< y_2$, $y_1<y_3$ and $W_i\cap W_j\neq\emptyset$, $\forall 1\leq i,\,j\leq 3$ if $W_i\neq\emptyset$ and $W_j\neq \emptyset$. So we can obtain finer bounds for the intersection parts by choosing the smaller relatively upper bounds.

On the other hand, $|b|$ and $|c|$ do not appear in the expression of gradient estimates and the definition of $W_i, i=1,2,3$ only depends on the sign of $b$ and $c$. We can re-divide the regions $W_i, i=1,2,3$ according to the signs of $b$ and $c$. Hence, we can get the following more obvious corollaries.

\begin{cor}\label{cor4.4}
Let $(M,g)$ be an $n$-dim($n\geq 2$) complete non-compact Riemannian manifold satisfying $\mathrm{Ric}_g\geq-(n-1)\kappa g$ for some constant $\kappa\geq0$. Let $v\in C^1(M)$ be an entire positive solution of \eqref{equa1} with $b>0, c>0$.
	\begin{enumerate}
		\item If $(q,r)\in V_1$, where
		\begin{align}\label{cond1b}
			 V_1=\left\{(q,r):   \frac{n+1}{n-1}-\frac{r}{p-1}\geq 0\right\},
		\end{align}
		then
		\begin{align*}
			\frac{|\nabla v|}{v}(x)\leq \frac{(n-1)\sqrt{\kappa}}{p-1}, \quad\forall x\in M.
		\end{align*}
		\item If $(q,r)\in V_2$, where
		\begin{align}\label{cond2b}
		V_2=\left\{(q,r):  \frac{n+1}{n-1}<\frac{r}{p-1}< \frac{n+3}{n-1}\right\},
		\end{align}
		then
		\begin{align*}
			\frac{|\nabla v|}{v}(x)\leq \frac{2\sqrt{\kappa}}{(p-1)\sqrt{\left(\frac{r}{p-1}-1\right)\left(\frac{n+3}{n-1}-\frac{r}{p-1}\right)}}, \quad\forall x\in M.
\end{align*}
\end{enumerate}
\end{cor}

\begin{cor}\label{cor4.5}
	Let $(M,g)$ be an $n$-dim($n\geq 2$) complete non-compact Riemannian manifold satisfying $\mathrm{Ric}_g\geq-(n-1)\kappa g$ for some constant $\kappa\geq0$. Let $v\in C^1(M)$ be an entire positive solution of \eqref{equa1} with $b<0$ and $c<0$.
	\begin{enumerate}
		\item If  $(q,r)\in V_1$, where
		\begin{align}\label{cond1c}
			 V_1=\left\{(q,r):  \frac{n+1}{n-1}-\frac{q}{p-1}\leq 0\right\},
		\end{align}
		then
		\begin{align*}
			\frac{|\nabla v|}{v}(x)\leq \frac{(n-1)\sqrt{\kappa}}{p-1}, \quad\forall x\in M.
		\end{align*}
		\item If $(q,r)\in V_2$, where
		\begin{align}\label{cond2c}
			 V_2=\left\{(q,r):  1<\frac{q}{p-1}<  \frac{n+1}{n-1}\right\},
		\end{align}
		then
		\begin{align*}
			\frac{|\nabla v|}{v}(x)\leq \frac{2\sqrt{\kappa}}{(p-1)\sqrt{\left(\frac{q}{p-1}-1\right)\left(\frac{n+3}{n-1}-\frac{q}{p-1}\right)}}, \quad\forall x\in M.
		\end{align*}
	\end{enumerate}
\end{cor}

\begin{figure}[h]
	\begin{tikzpicture}
		\pgfsetfillopacity{0.5}
		%\draw[help lines, color=red!5, dashed] (-0.5,-0.5) grid (3.2,3.2);
		\path[fill=red](-0.2, -0.2)--(2,2)--(-0.2, 2);
		%\path[fill=green](4/3,3.6)--(4/3,4/3)--(8/3,8/3)--(8/3, 3.6);
		\path[fill=yellow](2, 2)--(8/3,8/3)--(-0.2,8/3)--(-0.2, 2);
		%\fill[pattern=north west lines](3,3) circle (1);
		%\fill[gray](3,3) circle (1);
		%\path[pattern=north west lines](2, 2)--(7, 2)--(7, 5/2)--(5 ,5/2)--(5,5)--(2, 5);
		%\draw[pattern=north west lines] (1, 1) rectangle (6 ,5);
		%\draw[pattern=north west lines] (1, 1) rectangle (7 ,6/5);
		%\draw[pattern=fivepointed stars] (0.5,0.5) rectangle (3,1);
		\draw[dotted] (-0.2, -0.2) -- (3, 3);
		%\draw[dotted] (-0.2, 8/3 )--(2, 8/3) node[anchor=east]{$\frac{r}{p-1}= \frac{n+3}{n-1}$}-- (4, 8/3 );
		\draw[dotted] (-0.2, 2 )--(2, 2) node[anchor=east]{$\frac{r}{p-1}= \frac{n+1}{n-1}$}-- (3, 2 );
		%\draw[dotted] (-0.5, 4/3 )--(2, 4/3) node[anchor=east]{$\frac{r}{p-1}= 1 $}-- (4, 4/3 )  -- (5, 4/3);
		\draw[dotted] (-0.2, 8/3 )--(2.5, 8/3) node[anchor=east]{$\frac{r}{p-1}= \frac{n+3}{n-1}$}-- (3, 8/3 );
		%\draw[dotted] (-0.2, 4/3 )--(2, 4/3) node[anchor=east]{$\frac{r}{p-1}= 1$}-- (3, 4/3);
		%\draw[dotted] (-0.5,2) --(1,2)  node[anchor=north]{$\frac{r}{p-1}= 1$} -- (4,2)
		%\path[fill=yellow](-0.5, 1.8)--(3.5,1.8)--(3.5,-0.5)--(-0.5, -0.5);
		\draw (0.8,0.8) node[] {\rotatebox{45}{$q=r $}};
		\draw[->,ultra thick] (-0.2,0)--(3,0) node[right]{$\frac{q}{p-1}$};
		\draw[->,ultra thick] (0,-0.2)--(0,4) node[above]{$\frac{r}{p-1}$};
	\end{tikzpicture}			
	\begin{tikzpicture}
		\pgfsetfillopacity{0.5}
		\draw[help lines, color=red!5, dashed] (-0.5,-0.5) grid (4,4);
		\path[fill=red](2, 2)--(3.6,3.6)--(2,3.6);
		\path[fill=green](4/3,4/3)--(2,2)--(2,3.6)--(4/3, 3.6);
		%\fill[pattern=north west lines](3,3) circle (1);
		%\fill[gray](3,3) circle (1);
		%\path[pattern=north west lines](2, 2)--(7, 2)--(7, 5/2)--(5 ,5/2)--(5,5)--(2, 5);
		%\draw[pattern=north west lines] (1, 1) rectangle (6 ,5);
		%\draw[pattern=north west lines] (1, 1) rectangle (7 ,6/5);
		%\draw[pattern=fivepointed stars] (0.5,0.5) rectangle (3,1);
		\draw[dotted] (-0.2, -0.2) -- (3.5,3.5);
		\draw[dotted] (4/3,-0.2) --(4/3, 2.8)  node[anchor=north]{$\frac{q}{p-1}=1 $} -- (4/3,8/3) -- (4/3,3.6);
		%\draw[dotted] (-0.2, 8/3 )--(2, 8/3) node[anchor=east]{$\frac{r}{p-1}= \frac{n+3}{n-1}$}-- (4, 8/3 );
		%\draw[dotted] (-0.2, 2 )--(2, 2) node[anchor=east]{$\frac{r}{p-1}= \frac{n+1}{n-1}$}-- (3, 2 );
		%\draw[dotted] (8/3,-0.5) --(8/3, 1.3)  node[anchor=north]{$\frac{q}{p-1}= \frac{n+3}{n-1}$} -- (8/3,3.5);
		\draw[dotted] (2,-0.5) --(2, 1.3)  node[anchor=north]{$\frac{q}{p-1}= \frac{n+1}{n-1}$} -- (2,3.6);
		%\filldraw[black] (3,3) circle (1pt) node[anchor=south]{$(\frac{n+1}{n-1}, \frac{n+1}{n-1})$};
		%\draw[->] (3, 4.5) node[anchor=south]{$	1<\frac{r}{p-1}<\frac{n+3}{n-1},  1<\frac{q}{p-1}<\frac{n+2}{n},$}--(2.35, 3.2);
		%\path[fill=yellow](-0.5, 1.8)--(3.5,1.8)--(3.5,-0.5)--(-0.5, -0.5);
		\draw (0.8,0.8) node[] {\rotatebox{45}{$q=r $}};
		\draw[->,ultra thick] (-0.2,0)--(3,0) node[right]{$\frac{q}{p-1}$};
		\draw[->,ultra thick] (0,-0.2)--(0,4) node[above]{$\frac{r}{p-1}$};
	\end{tikzpicture}	
		\begin{tikzpicture}
		\pgfsetfillopacity{0.5}
		\draw[help lines, color=red!5, dashed] (-0.5,-0.5) grid (4,4);
		\path[fill=red](-0.2, 2)--(2,2)--(2,3.6)--(-0.2, 3.6);
		\path[fill=green](2,3.6)--(2,2)--(8/3,8/3)--(8/3, 3.6);
		\path[fill=yellow](4/3, 4/3)--(2,2)--(-0.2,2)--(-0.2, 4/3);
		%\fill[pattern=north west lines](3,3) circle (1);
		%\fill[gray](3,3) circle (1);
		\draw[dotted] (-0.2, -0.2) -- (3.2,3.2);
		%\draw[dotted] (4/3,-0.2) --(4/3, 2/3)  node[anchor=north]{$\frac{q}{p-1}=1 $} -- (4/3,8/3) -- (4/3,3);
		%\draw[dotted] (-0.2, 8/3 )--(2, 8/3) node[anchor=east]{$\frac{r}{p-1}= \frac{n+3}{n-1}$}-- (4, 8/3 );
		\draw[dotted] (-0.2, 2 )--(2, 2) node[anchor=east]{$\frac{r}{p-1}= \frac{n+1}{n-1}$}-- (3, 2 );
		\draw[dotted] (8/3,-0.5) --(8/3, 1.3)  node[anchor=north]{$\frac{q}{p-1}= \frac{n+3}{n-1}$} -- (8/3,3.5);
		\draw[dotted] (-0.2, 4/3 )--(1.5, 4/3) node[anchor=east]{$\frac{r}{p-1}= 1$}-- (3, 4/3);
		\draw (0.6,0.6) node[] {\rotatebox{45}{$q=r $}};
		\draw[->,ultra thick] (-0.2,0)--(3,0) node[right]{$\frac{q}{p-1}$};
		\draw[->,ultra thick] (0,-0.2)--(0,4) node[above]{$\frac{r}{p-1}$};
	\end{tikzpicture}					
\begin{tikzpicture}
		\path[fill=red!0](0, 1)--(0.5,1)--(0.5,1.1) --(0.5,1.2)--(0,1.2);		
		\path[fill=red!50](0, 2)--(0.5,2)--(0.5,2.1) node[anchor=west]{$f<y_1$}--(0.5,2.3)--(0,2.3);				
		\path[fill=green!50](0, 3)--(0.5,3)--(0.5,3.1) node[anchor=west]{$f<y_2$}--(0.5,3.3)--(0,3.3);
		\path[fill=yellow!50](0,4)--(0.5,4)--(0.5,4.1) node[anchor=west]{$f<y_3$}--(0.5,4.3)--(0,4.3);
\end{tikzpicture}
\caption{The region of $(\frac{q}{p-1}, \frac{r}{p-1})$ with different gradient bounds. The first figure shows the case $b>0,c>0$; the second figure shows the case $b<0, c<0$; the third figure shows the case $b>0, c<0$.}
\end{figure}
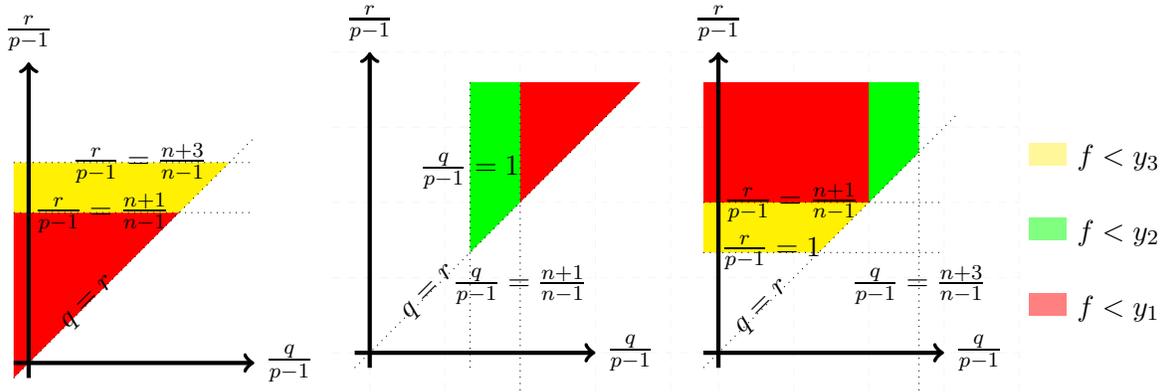

\begin{cor}\label{cor4.6}
Let $(M,g)$ be an $n$-dim($n\geq 2$) complete non-compact Riemannian manifold satisfying $\mathrm{Ric}_g\geq-(n-1)\kappa g$ for some constant $\kappa\geq0$. Assume that $v\in C^1(M)$ is an entire positive solution of \eqref{equa1} with $b>0$ and $c<0$.
\begin{enumerate}
\item If $(q,r)\in V_1$, where
\begin{align}\label{cond1d}
V_1=\left\{(q,r): \frac{n+1}{n-1}-\frac{q}{p-1}\geq 0\quad\text{and}\quad \frac{n+1}{n-1}-\frac{r}{p-1}\leq 0\right\},
\end{align}
then
\begin{align*}
\frac{|\nabla v|}{v}(x)\leq \frac{(n-1)\sqrt{\kappa}}{p-1}, \quad\forall x\in M.
\end{align*}
\item If $(q,r)\in V_2$, where
\begin{align}\label{cond2d}
V_2=\left\{(q,r): 1<\frac{r}{p-1}< \frac{n+1}{n-1}\right\},
\end{align}
then
\begin{align*}
\frac{|\nabla v|}{v}(x)\leq \frac{2\sqrt{\kappa}}{(p-1)\sqrt{\left(\frac{r}{p-1}-1\right)\left(\frac{n+3}{n-1}-\frac{r}{p-1}\right)}}, \quad\forall x\in M.
\end{align*}
		
\item If $(q,r)\in V_3$, where
\begin{align}\label{cond3d}
V_3=\left\{(q,r):    \frac{n+1}{n-1}<\frac{q}{p-1}< \frac{n+3}{n-1}
\right\},
\end{align}
then
\begin{align*}
\frac{|\nabla v|}{v}(x)\leq \frac{2\sqrt{\kappa}}{(p-1)\sqrt{\left(\frac{q}{p-1}-1\right)\left(\frac{n+3}{n-1}-\frac{q}{p-1}\right)}}, \quad\forall x\in M.
\end{align*}
\end{enumerate}
\end{cor}

Finally, we prove \corref{thm1.6} and \corref{thm1.7}. These two corollaries are obvious in light of Theorem \ref{t10}.

\subsection{Proof of \corref{thm1.6} and \corref{thm1.7}}\

In this case, $b=1$, $c=-1$, $q=1$ and $r=3$. Conditions \eqref{equa1.10}, \eqref{equa1.11}, \eqref{equa1.12} in \corref{thm1.6} respectively correspond to conditions \eqref{cond1d}, \eqref{cond2d} and \eqref{cond3d} in \corref{cor4.6}.
\corref{thm1.6} follows from \corref{cor4.6}.

The proof of \corref{thm1.7} is almost the same and we omit it.
\qed

From the above arguments on \corref{thm1.6} and \corref{thm1.7} one see that for $n\geq 2$ the following assumption
$$\frac{2n+2}{n+3}< p$$
can not be dropped. A natural problem arises: whether the above conclusions for Allen-Cahn equation and Fisher-KPP equation still hold true for
$$1 <p \leq \frac{2n+2}{n+3}$$
or not?
\medskip
	
\noindent {\it\bf{Acknowledgements}}: The author Y. Ma is supported by YSBR-001. The author Y. Wang is supported partially by National Key Research and Development projects of China (Grant No. 2020YFA0712500).
	%%%==========================
\medskip\medskip

\bibliographystyle{plain}

\begin{thebibliography}{10}
	
	\bibitem{MR1004713}
	Marie-Fran\c{c}oise Bidaut-V\'{e}ron.
	\newblock Local and global behavior of solutions of quasilinear equations of
	{E}mden-{F}owler type.
	\newblock {\em Arch. Rational Mech. Anal.}, 107(4):293--324, 1989.
	
	\bibitem{MR1134481}
	Marie-Fran\c{c}oise Bidaut-V\'{e}ron and Laurent V\'{e}ron.
	\newblock Nonlinear elliptic equations on compact {R}iemannian manifolds and
	asymptotics of {E}mden equations.
	\newblock {\em Invent. Math.}, 106(3):489--539, 1991.
	
	\bibitem{MR982351}
	Luis~A. Caffarelli, Basilis Gidas, and Joel Spruck.
	\newblock Asymptotic symmetry and local behavior of semilinear elliptic
	equations with critical {S}obolev growth.
	\newblock {\em Comm. Pure Appl. Math.}, 42(3):271--297, 1989.
	
	\bibitem{MR1121147}
	Wen~Xiong Chen and Congming Li.
	\newblock Classification of solutions of some nonlinear elliptic equations.
	\newblock {\em Duke Math. J.}, 63(3):615--622, 1991.
	
	\bibitem{MR0709038}
	Emmanuele DiBenedetto.
	\newblock {$C\sp{1+\alpha }$} local regularity of weak solutions of degenerate
	elliptic equations.
	\newblock {\em Nonlinear Anal.}, 7(8):827--850, 1983.
	
	\bibitem{MR4240763}
	Weiwei Ding, Yihong Du, and Zongming Guo.
	\newblock The {S}tefan problem for the {F}isher-{KPP} equation with unbounded
	initial range.
	\newblock {\em Calc. Var. Partial Differential Equations}, 60(2):Paper No. 69,
	37, 2021.
	
	\bibitem{fisher1937wave}
	Ronald~Aylmer Fisher.
	\newblock The wave of advance of advantageous genes.
	\newblock {\em Annals of eugenics}, 7(4):355--369, 1937.
	
	\bibitem{MR615628}
	Basilis Gidas and Joel Spruck.
	\newblock Global and local behavior of positive solutions of nonlinear elliptic
	equations.
	\newblock {\em Comm. Pure Appl. Math.}, 34(4):525--598, 1981.
	
	\bibitem{han2023gradient}
	Dong Han, Jie He, and Youde Wang.
	\newblock Gradient estimates for ${\Delta}_pu-|\nabla u|^q+b(x)|u|^{r-1}u=0$ on
	a complete riemannian manifold and liouville type theorems.
	\newblock {\em arXiv:2309.03510}, pages 1--37, 2023.
	
	\bibitem{hewangwei2024}
	Jie He, Youde Wang, and Guodong Wei.
	\newblock Gradient estimate for solutions of the equation ${\Delta}_pv +av^q=0$
	on a complete riemannian manifold.
	\newblock {\em Math. Z.}, 306(3):Paper No. 46, 19, 2024.
	
	\bibitem{HW}
	Pingliang Huang and Youde Wang.
	\newblock Gradient estimates and {L}iouville theorems for {L}ichnerowicz
	equations.
	\newblock {\em Pacific J. Math.}, 317(2):363--386, 2022.
	
	\bibitem{MR2518892}
	Brett Kotschwar and Lei Ni.
	\newblock Local gradient estimates of {$p$}-harmonic functions, {$1/H$}-flow,
	and an entropy formula.
	\newblock {\em Ann. Sci. \'{E}c. Norm. Sup\'{e}r. (4)}, 42(1):1--36, 2009.
	
	\bibitem{MR2962229}
	Peter Li.
	\newblock {\em Geometric analysis}, volume 134 of {\em Cambridge Studies in
		Advanced Mathematics}.
	\newblock Cambridge University Press, Cambridge, 2012.
	
\bibitem{MR0834612}Peter Li and Shing-Tung Yau.
\newblock On the parabolic kernel of the {S}chr\"{o}dinger operator.
\newblock {\em Acta Math.}, 156(3-4): 153--201, 1986.
	
\bibitem{Li-Zhang}YanYan Li and Lei Zhang.
\newblock Liouville-type theorems and {H}arnack-type inequalities for semilinear elliptic equations.
\newblock {\em J. Anal. Math.}, 90:27--87, 2003.
	
\bibitem{MR829846}Wei-Ming Ni and James Serrin.
\newblock Nonexistence theorems for singular solutions of quasilinear partial differential equations.
\newblock{\em Comm. Pure Appl. Math.}, 39(3): 379--399, 1986.
	
\bibitem{PWW1}Bo~Peng, Youde Wang, and Guodong Wei.
\newblock Gradient estimates for ${\Delta}u + au^{p+1}=0$ and liouville theorems.
\newblock {\em Mathematical Theory and Applications}, 43(1): 32--43, 3 2023.
	
\bibitem{MR2350853}Peter Pol\'{a}\v{c}ik, Pavol Quittner, and Philippe Souplet.
\newblock Singularity and decay estimates in superlinear problems via {L}iouville-type theorems. {I}. {E}lliptic equations and systems.
\newblock {\em Duke Math. J.}, 139(3): 555--579, 2007.
	
\bibitem{saloff1992uniformly}Laurent Saloff-Coste.
\newblock Uniformly elliptic operators on riemannian manifolds.
\newblock {\em Journal of Differential Geometry}, 36(2):417--450, 1992.
	
\bibitem{MR1333601}Richard Schoen and Shing-Tung Yau.
\newblock {\em Lectures on differential geometry}.
\newblock Conference Proceedings and Lecture Notes in Geometry and Topology, I. International Press, Cambridge, MA, 1994.
\newblock Lecture notes prepared by Wei Yue Ding, Kung-Ching Chang [Gong Qing Zhang], Jia-Qing Zhong and Yi-Chao Xu, Translated from the Chinese by Ding and S. Y. Cheng, With a preface translated from the Chinese by Kaising Tso.
	
\bibitem{MR1946918}James Serrin and Henghui Zou.
\newblock Cauchy-{L}iouville and universal boundedness theorems for quasilinear elliptic equations and inequalities.
\newblock {\em Acta Math.}, 189(1):79--142, 2002.
	
\bibitem{MR2522424}Philippe Souplet.
\newblock The proof of the {L}ane-{E}mden conjecture in four space dimensions.
\newblock {\em Adv. Math.}, 221(5):1409--1427, 2009.
	
\bibitem{MR3275651}Chiung-Jue~Anna Sung and Jiaping Wang.
\newblock Sharp gradient estimate and spectral rigidity for {$p$}-{L}aplacian.
\newblock {\em Math. Res. Lett.}, 21(4):885--904, 2014.
	
\bibitem{MR0727034}Peter Tolksdorf.
\newblock Regularity for a more general class of quasilinear elliptic equations.
\newblock {\em J. Differential Equations}, 51(1):126--150, 1984.
	
\bibitem{MR0474389}Karen Uhlenbeck.
\newblock Regularity for a class of non-linear elliptic systems.
\newblock {\em Acta Math.}, 138(3-4):219--240, 1977.
	
\bibitem{MR2880214}Xiaodong Wang and Lei Zhang.
\newblock Local gradient estimate for {$p$}-harmonic functions on {R}iemannian
	manifolds.
\newblock{\em Comm. Anal. Geom.}, 19(4):759--771, 2011.

\bibitem{Wang} Youde Wang.
\newblock Harmonic maps from noncompact Riemannian manifolds with non-negative Ricci curvature outside a compact set.
\newblock {\em Proc. Roy. Soc. Edinburgh Sect. A}, 124 (1994), no. 6, 1259-1275.
	
\bibitem{MR4559367} Youde Wang and Guodong Wei.
\newblock On the nonexistence of positive solution to {$\Delta u + au^{p+1} = 0$} on {R}iemannian manifolds.
\newblock {\em J. Differential Equations}, 362:74--87, 2023.
	
\bibitem{WZ1} Youde Wang and Aiqi Zhang.
\newblock Gradient estimate for solutions of $\delta v+ v^r-v^s= 0$ on a complete riemannian manifold.
\newblock {\em Mathematical Theory and Applications}, 43(3):1--22, 9 2023.
	
\bibitem{MR3268873} Chao Xia.
\newblock Local gradient estimate for harmonic functions on {F}insler manifolds.
\newblock {\em Calc. Var. Partial Differential Equations}, 51(3-4):849--865, 2014.
	
\bibitem{MR431040} Shing-Tung Yau.
\newblock Harmonic functions on complete {R}iemannian manifolds.
\newblock {\em Comm. Pure Appl. Math.}, 28:201--228, 1975.
	
\bibitem{MR2981845}Hui-Chun Zhang and Xi-Ping Zhu.
\newblock Yau's gradient estimates on {A}lexandrov spaces.
\newblock {\em J. Differential Geom.}, 91(3):445--522, 2012.
	
\bibitem{MR3912761}Liang Zhao.
\newblock Liouville theorem for weighted {$p$}-{L}ichnerowicz equation on smooth metric measure space.
\newblock {\em J. Differential Equations}, 266(9):5615--5624, 2019.
	
\bibitem{MR3866881}Liang Zhao and Dengyun Yang.
\newblock Gradient estimates for the {$p$}-{L}aplacian {L}ichnerowicz equation on smooth metric measure spaces.
\newblock {\em Proc. Amer. Math. Soc.}, 146(12):5451--5461, 2018.
\end{thebibliography}

\end{document}